\documentclass[reqno]{amsart}
\usepackage[dvipsnames]{xcolor}
\usepackage[utf8]{inputenc}
\usepackage[shortlabels]{enumitem}
\usepackage{palatino,amsmath,amsfonts,amsthm,amssymb,mathrsfs,fullpage,soul,cite,tikz,float}
\usetikzlibrary{cd}
\usepackage{hyperref}
\hypersetup{colorlinks,citecolor=red,filecolor=black,linkcolor=blue,urlcolor=blue}
\usetikzlibrary{patterns}

\usepackage{accents}

\DeclareMathOperator{\Comp}{Comp} 
\DeclareMathOperator{\Con}{Con} 

\DeclareMathOperator{\Supp}{Supp} 
\DeclareMathOperator{\Min}{Min}
\DeclareMathOperator{\Max}{Max}
\DeclareMathOperator{\FInf}{FI}
\DeclareMathOperator{\sk}{sk} 
\DeclareMathOperator{\Fund}{Fund}

\newcommand{\isom}{\cong}
\newcommand{\x}{\times}
\newcommand{\defterm}[1]{\textbf{#1}}

\newcommand{\Set}{\mathsf{Set}}
\newcommand{\Setx}{\mathsf{Set}^\times}
\renewcommand{\Vec}{\mathsf{Vec}}

\newcommand{\rmG}{\mathrm{G}}		
\newcommand{\rmH}{\mathrm{H}}		\newcommand{\bfH}{\mathbf{H}}
\newcommand{\rmP}{\mathrm{P}}		\newcommand{\bfP}{\mathbf{P}}
\newcommand{\calH}{\mathcal{H}}
\newcommand{\calI}{\mathcal{I}}

\newcommand{\SFam}{\mathrm{SF}}		\newcommand{\SFAM}{\mathbf{SF}}
\newcommand{\Acc}{\mathrm{Acc}}			
\newcommand{\Int}{\mathrm{Int}}			
\newcommand{\Union}{\mathrm{Union}}		
\newcommand{\Simp}{\mathrm{Simp}}		\newcommand{\SIMP}{\mathbf{Simp}}
\newcommand{\Mat}{\mathrm{Mat}}			
\newcommand{\Cf}{\mathrm{CF}}			\newcommand{\CF}{\mathbf{CF}}
\newcommand{\CFA}{\mathcal{CF}} 
\newcommand{\Loi}{\mathrm{LOI}}			\newcommand{\LOI}{\mathbf{LOI}}
\newcommand{\NTop}{\mathrm{NTop}}		
\newcommand{\AMat}{\mathrm{AMat}}		
\newcommand{\Bool}{\mathrm{Bool}}		\newcommand{\BOOL}{\mathbf{Bool}}

\newcommand{\pha}{\mathsf{F}}
\newcommand{\cg}[2]{\mathsf{C}_{#1}\pha^{#2}}
\newcommand{\chain}[1]{\mathsf{C}_{#1}}

\newcommand{\fock}{\bar{\mathcal{K}}}

\newtheorem{theorem}{Theorem}[section]
\newtheorem{proposition}[theorem]{Proposition}

\newtheorem{corollary}[theorem]{Corollary}

\theoremstyle{definition}
\newtheorem{definition}[theorem]{Definition}
\newtheorem{example}[theorem]{Example}
\newtheorem{remark}[theorem]{Remark}

\newcommand{\Cc}{\mathbb{C}}

\newcommand{\Rr}{\mathbb{R}}
\newcommand{\Xx}{\mathbb{X}}
\newcommand{\Ss}{\mathbb{S}} 

\newcommand{\sss}{\mathbf{s}}
\newcommand{\ttt}{\mathbf{t}}

\newcommand{\F}{\mathcal{F}}

\newcommand{\0}{\emptyset}
\newcommand{\partn}{\vdash}
\newcommand{\compn}{\models}
\newcommand{\sm}{\setminus}
\newcommand{\filter}[1]{\lceil #1 \rceil}
\newcommand{\ideal}[1]{\lfloor #1 \rfloor}
\newcommand{\anti}{\mathsf{S}} 

\newcommand{\Ph}{P_{\rm{hi}}}
\newcommand{\Pl}{P_{\rm{lo}}}
\newcommand{\Qh}{Q_{\rm{hi}}}
\newcommand{\Ql}{Q_{\rm{lo}}}

\DeclareMathOperator{\Good}{Good}
\DeclareMathOperator{\Acyc}{Acyc}
\DeclareMathOperator{\Hyb}{Hyb}

\newcommand{\Red}[1]{{\color{red}{#1}}}
\newcommand{\Blue}[1]{{\color{blue}{#1}}}
\newcommand{\Green}[1]{{\color{ForestGreen}{#1}}}

\newcommand{\gr}[1]{\|#1\|} 
\newcommand{\decomp}[2]{{#1}^{[#2]}} 

\author{Kevin Marshall}
\address{Department of Mathematics, University of Kansas, Lawrence, Kansas 66045, USA}
\email{kmarsh729@gmail.com}
\author{Jeremy L.\ Martin}
\thanks{JLM was partially supported by Simons Collaboration Grant \#315347.}
\address{Department of Mathematics, University of Kansas, Lawrence, Kansas 66045, USA}
\email{jlmartin@ku.edu}
\title{Hopf Monoids of Set Families}
\date{\today}
\begin{document}

\begin{abstract}
A \textit{grounded set family} on $I$ is a subset $\mathcal{F}\subseteq2^I$ such that $\emptyset\in\mathcal{F}$.  We study a linearized Hopf monoid $\mathbf{SF}$ on grounded set families, with restriction and contraction inspired by the corresponding operations for antimatroids.  Many known combinatorial species, including simplicial complexes and matroids, form Hopf submonoids of $\mathbf{SF}$, although not always with the ``standard'' Hopf structure (for example, our contraction operation is not the usual contraction of matroids).  We use the topological methods of Aguiar and Ardila to obtain a cancellation-free antipode formula for the Hopf submonoid of lattices of order ideals of finite posets.  Furthermore, we prove that the Hopf algebra of lattices of order ideals of chain forests (disjoint unions of chains) extends the Hopf algebra of symmetric functions, and that its character group extends the group of formal power series in one variable with constant term 1 under multiplication.
\end{abstract}

\maketitle

\section{Introduction}

Hopf monoids arise in combinatorics as collections of labeled objects that can be joined into larger objects and split into smaller objects of the same type in a coherent fashion.  There are standard and well-known Hopf monoid structures on objects such as matroids, graphs, and posets; see \cite{AA}.  In this paper we study a linearized Hopf monoid $\SFAM$ whose underlying objects are much more general: they are merely families $\F$ of subsets of a finite ground set $I$, with the only requirement that they be \textit{grounded}, i.e., $\0\in\F$.  Among other familiar objects, matroids and (finite) posets give rise to submonoids of $\SFAM$, since one can encode a matroid by its independence complex, and a poset by its family of order ideals.  However, the Hopf monoid structures on matroids and posets are not the same as the more familiar ones described in \cite{AA}, as we now explain.

The original motivation of this research was to construct a Hopf monoid structure on \textit{antimatroids}.  An antimatroid on ground set $I$ is a set family $\F\subseteq 2^I$ that is \textit{accessible} (for every nonempty $A\in\F$, there exists $x\in A$ such that $A\sm\{x\}\in\F$) and closed under union.  Just as matroids provide a combinatorial model of linear independence, antimatroids model convexity: specifically, if $I$ is a set of points in Euclidean space, then the family $\{A\subseteq I\colon \text{conv}(I\sm A)\cap A=\0\}$ is an antimatroid.  Standard sources on antimatroids include \cite{EJ} and (in the context of greedoids) \cite{greedoids}.  We define the product of antimatroids $\F_1$ and $\F_2$ on disjoint ground sets as the join
\[\F_1*\F_2 = \{X\cup Y \colon X\in\F_1,\ Y\in\F_2\}.\]
For an antimatroid $\F$ on ground set $I$, we define the restriction and contraction\footnote{See Remark~\ref{rem:terminology} for a note on terminology.} with respect to a decomposition $I=S\sqcup T$ by
\[
\F\vert_{S} = \{F \cap S \colon F \in \F\},\qquad
\F/_S = \{F \in \F\colon F\cap S = \0\}.
\]
Indeed, these operations define a Hopf monoid structure on antimatroids. But the antimatroid assumption is not necessary: we obtain a Hopf monoid on the species of all grounded set families.  Many familiar combinatorial objects arise as submonoids; a family tree is shown in Figure~\ref{fig:hierarchy}.  We focus in particular on three submonoids:
\begin{itemize}
\item $\LOI$ (lattices of order ideals of posets);
\item $\SIMP$ (simplicial complexes);
\item $\CF$ (lattices of order ideals of chain forests).
\end{itemize}

Every connected Hopf monoid admits an \textit{antipode} map $\anti$, defined by a certain commutative diagram \cite[p.11]{AguiarMahajan}.  The general closed formula for the antipode, known as the \textit{Takeuchi formula}, is in general highly non-cancellation-free, so a central problem for a given Hopf monoid is to simplify the Takeuchi formula.  Aguiar and Ardila~\cite[\S1.6]{AA} described a topological method for simplifying the antipode formula by interpreting its coefficients topologically as Euler characteristics of subfans of the braid arrangement; they accomplished this for the Hopf monoid of generalized permutahedra.  There does not appear to be a clear cancellation-free formula valid on all of~$\SFAM$.  On the other hand, the submonoid $\LOI$ \textit{does} admit a cancellation-free formula, which we now describe.

Let $P$ be a finite poset, and let $J(P)$ be its lattice of order ideals (which is a grounded set family).  A \textit{fracturing} $Q$ of $P$ is a disjoint union of induced subposets of $P$, whose union need not be all of $P$.  It is relatively easy to see that every term in the Takeuchi formula for $\anti(P)$ is a fracturing of $P$.  The more difficult task is to identify the \textit{good fracturings}, i.e., those that actually appear in the antipode, and to compute their coefficients.  The first question is answered by Prop.~\ref{classify-good-fracturings}: a fracturing $Q$ is good if it contains all minimal elements of $P$ and if the ordering of $P$ induces an ordering on the Hasse components of $Q$.  Following the method of \cite{AA}, the coefficient of $J(Q)$ can be interpreted as an Euler characteristic; the fan in question is not convex, but can be further decomposed as a union of convex fans, so that its Euler characteristic can be computed by inclusion/exclusion.  The result is a cancellation-free formula for the antipode (Theorem~\ref{thm:antipode-LOI}):
\[
\anti(J(P)) = \sum_{\text{good fracturings }Q} (-1)^{c(Q)+|P\sm Q|} J(Q)
\]
where $c(Q)$ is the number of Hasse components of $Q$.  Furthermore, we give a cancellation-free formula (Theorem~\ref{thm:antipode-of-ordinal-sum}) for the antipode of the ordinal sum of two posets.  The results on $\LOI$ appear in Section~\ref{sec:antipode-LOI} of the paper.

Section~\ref{sec:submonoid-simp} studies the Hopf monoid $\SIMP$ of simplicial complexes, which is the maximal cocommutative Hopf submonoid of $\SFAM$.  Here both restriction and contraction reduce to taking induced subcomplexes, giving the same coproduct on simplicial complexes studied by Benedetti, Hallam and Machacek~\cite{BHM} (see also~\cite[S5.3]{AA}), but with a different product, namely join instead of disjoint union.  The antipode on $\SIMP$ is not multiplicity-free, and appears intractable to compute in general, although we do give an explicit formula for skeletons of simplices (Theorem~\ref{thm:simplex-skeleton}).

Section~\ref{sec:chainforest} focuses on the Hopf monoid $\CF$ of chain forests, or rather the Hopf algebra $\CFA$ obtained from it by applying a Fock functor.  We show that $\CFA$ is an extension of the well-known Hopf algebra of symmetric functions (Theorem~\ref{thm:punchline}).  Moreover, its group of characters is an extension of the multiplicative group of formal power series in one variable with constant coefficient~1 (Theorem~\ref{thm:power-series-subgroup}), so that the antipode corresponds to reciprocation of power series.

\section{Background} \label{sec:background}

For a poset $P$ and a subset $A$, we write $\ideal{A}_P$ and $\filter{A}_P$ for, respectively, the order ideal and order filter generated by $A$ (dropping the subscript when no ambiguity can arise).

\subsection{Set compositions and their geometry} \label{sec:posets}

Here we summarize the correspondence between the combinatorics of (pre)posets and the geometry of the braid arrangement, for which the principal source is~\cite[\S3]{GP}.

A \defterm{set composition} of a finite set $I$ is an ordered partition $\Phi$ of $I$ into pairwise-disjoint sets $\Phi_1,\dots,\Phi_d$, called its \defterm{blocks}.  We write $\Phi\compn I$ or $\Phi\in\Comp(I)$.  Typically, $I=[n]=\{1,\dots,n\}$; in this case we can express $\Phi$ concisely by writing the blocks as sequences of digits, separated by bars: e.g., $\Phi=14|2|36|5$.  If $i\in\Phi_a$ and $j\in\Phi_b$, we write $i<_\Phi j$, $i>_\Phi j$, or $i=_\Phi j$ according as $a<b$, $a>b$, or $a=b$.  For instance, if $\Phi=14|2|36|5$ then $3=_\Phi6$ and $4<_\Phi3$.  The \defterm{(open) geometric realization} of $\Phi$ is the polyhedron
\[\gr{\Phi}=\{(x_1,\dots,x_n)\in\Rr^n \colon x_i<x_j \iff i<j\},\]
which is in fact a cone of the fan defined by the braid arrangement in $\Rr^n$ (for short, of the \defterm{braid fan}).  In particular $\dim\gr{\Phi}=|\Phi|$, the number of blocks of~$\Phi$.

A \defterm{preposet} is a binary relation $\leq$ that is reflexive and transitive, but not necessarily antisymmetric.  A preposet on $[n]$ can be regarded as a partial order on the blocks of a partition of $[n]$, as in Figure~\ref{fig:preposet}.
A linear preposet (one in which all elements are comparable) is thus equivalent to a set composition.  Moreover, we may define a linear extension of a preposet just as for a poset.
\begin{figure}[hb]
\begin{center}
\begin{tikzpicture}
\draw(0,0)--(1,1)--(2,0)--(3,1);
\foreach \x/\y/\l in {0/0/24, 1/1/1, 2/0/5, 3/1/36} \node[fill=white] at (\x,\y){\l};
\begin{scope}[shift={(8,0)}]
\draw(0,-1)--(0,2);
\foreach \x/\y/\l in {0/-1/24, 0/2/1, 0/0/5, 0/1/36} \node[fill=white] at (\x,\y){\l};
\end{scope}
\end{tikzpicture}
\end{center}
\caption{(Left) A preposet $Q$ on $[6]$.  (Right) A linear extension $Q'$ of $Q$.\label{fig:preposet}}
\end{figure}
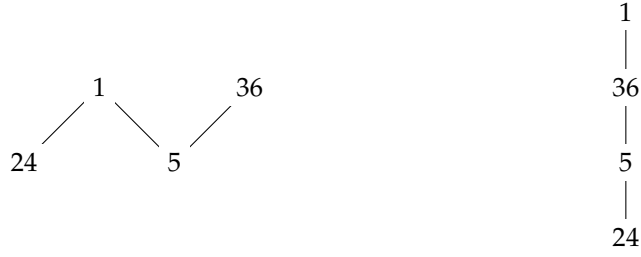

The (open) geometric realization $\gr{P}$ of a preposet $P$ on $[n]$ is the polyhedron in $\Rr^n$ defined by equalities and inequalities corresponding to the relations of $P$.  For instance, the preposet $Q$ shown in Figure~\ref{fig:preposet} corresponds to the polyhedron
\[\{(x_1,\dots,x_6)\in\Rr^6\colon x_2=x_4<x_1,\ x_5<x_1,\ x_5<x_3=x_6\}.\]
We may also define $\gr{P}$ as the interior of the union of the closures of the geometric realizations of preposet linear extensions of $P$.  Accordingly, geometric realization gives a bijection between preposets and convex unions of cones of the braid arrangement.  (These objects are called ``braid cones'' in \cite{GP}, but we reserve that term for single cones of the braid arrangement.)

The intersections of the braid cones with the $(n-2)$-sphere
\[\left\{(x_1,\dots,x_n)\in\Rr^n\colon \textstyle\sum_{i=1}^nx_i=0,\ \sum_{i=1}^nx_i^2=1\right\}.\]
form a triangulation of that sphere.  Therefore, collections of braid cones can often be viewed as (relative) simplicial complexes, e.g., for the purpose of calculating Euler characteristics.

\subsection{Hopf monoids} \label{sec:Hopf}

Next we sketch the theory of Hopf monoids in species, as described in more detail in~\cite{AguiarMahajan} and~\cite{AA}.

Let $\Set$ denote the category of sets with arbitrary functions and $\Setx$ the category of finite sets with bijections.
A \defterm{set species} is a covariant functor $\rmP:\Setx\to\Set$, sending a finite set $I$ to a set $\rmP[I]$.  Typically, $\rmP[I]$ is a set of combinatorial objects labeled by ground set $I$.  Functoriality means that every bijection $I\to J$ can be thought of as a relabeling of the ground set, which naturally induces a relabeling of the combinatorial objects in question.  We shall always work with set species that are \defterm{connected}, i.e., for which $|\rmP[\0]|=1$.

A \defterm{Hopf monoid in set species} is a set species $\rmH$ equipped with \defterm{product} and \defterm{coproduct} maps
\begin{equation} \label{hopferations}
\begin{aligned}
\mu_{S,T}:\rmH[S] \times \rmH[T] &\to \rmH[I] &&&&&
\Delta_{S,T}:\rmH[I] &\to \rmH[S] \times \rmH[T]\\
(x,y) &\mapsto x\cdot y &&&&&
z &\mapsto (z|_S,\;z/_S)
\end{aligned}
\end{equation}
(for every decomposition $I = S \sqcup T$), satisfying the axioms of naturality, unitality, associativity, coassociativity, and compatibility; see~\cite[pp.16--17]{AA} for the details.  In particular, associativity and coassociativity say that product and coproduct can be iterated: for every ordered set composition $\Phi=\Phi_1|\cdots|\Phi_d\compn I$, there are well-defined operations
\begin{equation} \label{hopferations-iterated}
\begin{aligned}
\mu_\Phi:\rmH[\Phi_1] \times\cdots\times \rmH[\Phi_d] &\to \rmH[I], &&&&&
\Delta_\Phi:\rmH[I] &\to \rmH[\Phi_1] \times\cdots\times \rmH[\Phi_d].
\end{aligned}
\end{equation}

In the connected case, the element of $\rmH[\0]$ is the multiplicative unit.  The Hopf monoid may or may not be commutative or cocommutative.  (We say that $\rmH[\0]$ is \defterm{cocommutative} if, whenever $I=S\sqcup T$ and $z\in\rmH[i]$, then $(z|_S,\;z/_S)=(z/_T,\;z|_T)$.) The elements $z|_S$ and $z/_S$ are called the \defterm{restriction} of $z$ to $S$ and the \defterm{contraction} of $z$ by $S$, respectively.  A \defterm{submonoid} of $\rmH$ is a subspecies $\rmG$ (i.e., a species $\rmG$ such that $\rmG[I] \subset \rmH[I]$ for all $I$) closed under product and coproduct.  It is easily checked that the intersection of submonoids is a submonoid.

Now fix a field $\Bbbk$ and let $\Vec$ be the category of $\Bbbk$-vector spaces with linear transformations.  A \defterm{vector species} is a covariant functor $\bfP:\Setx\to\Vec$, sending $I$ to $\bfP[I]$.  A \defterm{Hopf monoid in vector species} is a vector species $\bfH$ equipped with $\Bbbk$-linear product and coproduct operations
\begin{equation} \label{hopferations-vec}
\begin{aligned}
\mu_{S,T}:\bfH[S] \otimes \bfH[T] &\to \bfH[I] &&&&&
\Delta_{S,T}:\bfH[I] &\to \bfH[S] \otimes \bfH[T]\\
x\otimes y &\mapsto x\cdot y &&&&&
z &\mapsto z|_S\,\otimes\,z/_S
\end{aligned}
\end{equation}
satisfying the linear versions of the axioms of naturality, unitality, associativity, coassociativity, and compatibility.  We shall always assume that Hopf monoids in vector species are connected, i.e., $\bfH[\0]\isom\Bbbk$; the unit element of $\Bbbk$ behaves as a multiplicative unit on $\bfH$.

The \defterm{linearization} of a set species $\rmP$ is the vector species $\bfP=\Bbbk(\rmP)$, where $\bfP[I]$ is the $\Bbbk$-vector space with basis $\rmP[I]$.  Similarly, the linearization of a Hopf monoid in set species $\rmP$ is the vector species $\Bbbk(\rmP)$, with product and coproduct defined by linearly extending those of $\rmP$.

Hopf monoids can be regarded as generalizations of groups \cite[Ex.~1.31]{GR}, \cite[p.88]{AguiarMahajan}; the analogy of group inversion is the \defterm{antipode} of a Hopf monoid.  The antipode is a collection of maps $\anti_I:\bfH[I]\to\bfH[I]$ defined by a certain commutative diagram \cite[Defn.~1.15]{AguiarMahajan}, or equivalently by \textit{Takeuchi's formula}
\cite[Defn.~1.1.11]{AA}:
\begin{equation} \label{Takeuchi}
\anti_I = \sum_{\Phi \compn I} (-1)^{|\Phi|} \mu_\Phi\circ\Delta_\Phi.
\end{equation}
While general and explicit, Takeuchi's formula typically produces many similar terms, so for a specific Hopf monoid one would like to find a cancellation-free formula.  An important example is Aguiar and Ardila's cancellation-free antipode formula for the Hopf monoid of generalized permutahedra
\cite[\S1.6]{AA}, obtained by combining like terms and interpreting the coefficients as Euler characteristics of subcomplexes of a triangulated sphere, using the geometry described in \S\ref{sec:posets}.

\section{The Hopf monoid structure on grounded set families}\label{sec:setfam}

A \defterm{set family} is a pair $(\F,I)$, where $I$ is a finite ground set and $\F\subseteq 2^I$.  For brevity, we often drop the reference to the ground set~$I$ when it is clear from context.  We will assume that all set families are \defterm{grounded}, i.e., $\0\in\F$.  On the other hand, we do \textit{not} assume that each element of $I$ must belong to some set in $\F$; an element that does not do so is called a \defterm{phantom}.  Phantoms are somewhat analogous to loops in a graph.  They arise naturally from operations such as contraction (see Remark~\ref{rmk:phantom}) and so cannot be ignored entirely, but typically removing them has a very simple effect (e.g., Proposition~\ref{prop:exorcism} below).

Define a set species $\SFam$ by letting $\SFam[I]$ be the set of grounded set families $(\F,I)$.  In fact $\SFam$ is connected, since the unique grounded set family on $\0$ is $\{\0\}$.  In addition, we set $\SFAM=\Bbbk(\SFam)$.  (As a warning to the reader, these species have no connection whatsoever with the Hopf monoids of \textit{submodular functions} discussed in \cite[\S3.1]{AA} which are denoted by the same symbols.)

We now construct a Hopf product and coproduct on $\SFam$.  For $I_1\cap I_2=\0$, the \defterm{join} of grounded set families $(\F_1,I_1)$ and $(\F_2,I_2)$ is $(\F_1*\F_2,I_1\sqcup I_2)$, where
\[
\F_1*\F_2 = \{X \cup Y\colon X \in \F_1, Y \in \F_2\}.
\]
(This choice of join justifies our insistence on groundedness: the trivial family $\{\0\}$ is a two-sided identity for join.)
For $(\F,I)$ a set family and $S\subseteq I$, the \defterm{restriction} $\F|_S$ and the \defterm{contraction} $\F/_S$ are the set families
\begin{equation} \label{restrict-contract}
\F\vert_{S} = \{F \cap S \colon F \in \F\},\qquad
\F/_S = \{F \in \F\colon F\cap S = \0\},
\end{equation}
with ground sets $S$ and $I\sm S$, respectively.

\begin{remark} \label{rem:terminology}
We have chosen these terms in parallel with the more familiar operations of those names on matroids (in particular, this restriction operation is precisely restriction for matroids).  What we call ``restriction'' and ``contraction'' are called respectively ``trace'' and ``restriction'' in \cite{greedoids} (p.325 and p.292, respectively).  Meanwhile, \cite{EJ} uses ``restriction'' in the same sense as \cite{greedoids}, but uses contraction for a different operation that is not defined for every decomposition $S\cup T$.
\end{remark}

\begin{remark} \label{rmk:phantom}
If $(\F,I)$ is phantom-free, then so are all its restrictions, but not necessarily its contractions: for instance, if $I=\{1,2\}$, $\F=\{\0,\{1\},\{1,2\}\}$, and $S=\{1\}$, then $\F/_S$ is the trivial family $\{\0\}$ on ground set $\{2\}$.  This is why we must allow the possibility of phantoms.
\end{remark}

\begin{theorem} \label{thm:SFam-is-HM}
The set species $\SFam$ admits the structure of a commutative Hopf monoid, with product $\F_1\cdot\F_2=\F_1 * \F_2$ and coproduct
$\Delta_{S,T}(\F) = (\F\vert_S,\F/_S)$.
\end{theorem}

The proof is routine and technical, so we only sketch it here.  Unitality and naturality are immediate from the definitions, and associativity and commutativity of the product follows from associativity and commutativity of union.  Coassociativity and compatibility are checked by straightforward calculations.  For later use, we observe that the iterated coproduct operation with respect to a set composition $\Phi$ is given by
$\Delta_\Phi(\F) = (\F_1, \dots, \F_m)$, where
\begin{equation} \label{iterated-coproduct}
\F_i = \{A \cap \Phi_i\colon A \in \F \text{ and } A\cap\Phi_j = \0\ \ \forall j<i\}.
\end{equation}

Before going further, we calculate the effect on a set family of adding or removing a phantom, which is essentially to change the sign of its antipode.  Let $I$ be a finite set and $x \notin I$.  Let $\gamma:\SFAM[I]\to\SFAM[I\cup\{x\}]$ be the linear map defined on standard basis elements by $\gamma(\mathcal{F},I)=(\mathcal{F},I\cup\{x\})$.  Thus $\gamma$ is an isomorphism from $\SFAM[I]$ to the vector space spanned by set families on $I\cup\{x\}$ in which $x$ is a phantom.

\begin{proposition}\label{prop:exorcism}
With the foregoing setup, $\anti_{I\cup\{x\}}\circ\gamma = -\gamma\circ\anti_I$.
\end{proposition}

\begin{proof}
Abbreviate $I'=I\cup\{x\}$.
Let $C_1=\{\Phi\compn I'\colon\{x\}\in\Phi\}$.  Define a map $f:C_1\to\Comp(I)$ by removing the block $\{x\}$.  Then $f(\Phi)$ is a composition with $|\Phi|-1$ blocks, and $|f^{-1}(\Psi)|=|\Psi|+1$ for each~$\Psi\in\Comp(I)$.

Now let $C_2=\Comp(I')\sm C_1$, and define $g:C_2\to\Comp(I)$ by deleting $x$ from the (non-singleton) block containing it.  Then $g(\Phi)$ is a composition with $|\Phi|$ blocks, and $|g^{-1}(\Psi)|=|\Psi|$ for each~$\Psi\in\Comp(I)$.

Now, starting with Takeuchi's formula,
\begin{align*}
\anti_{I'}(\gamma(\mathcal{F},I)) = \anti_{I'}(\mathcal{F},I')
&= \sum_{\Phi\in C_1} (-1)^{|\Phi|} \mu_\Phi(\Delta_\Phi(\mathcal{F},I'))
+ \sum_{\Phi\in C_2} (-1)^{|\Phi|} \mu_\Phi(\Delta_\Phi(\mathcal{F},I'))\\
&= \sum_{\Psi \compn I} (|\Psi|+1)(-1)^{|\Psi|+1} \gamma(\mu_\Psi(\Delta_\Psi(\mathcal{F},I))) + \sum_{\Psi \compn I} |\Psi|(-1)^{|\Psi|} \gamma(\mu_\Psi(\Delta_\Psi(\mathcal{F},I)))\\
&= -\sum_{\Psi \compn I} (-1)^{|\Psi|} \gamma(\mu_\Psi(\Delta_\Psi(\mathcal{F},I)))\\
&= -\gamma(\anti_I(\mathcal{F},I)).\qedhere
\end{align*}
\end{proof}

Let $\F$ be a set family.  We say that:
\begin{itemize}
\item $\F$ is \defterm{accessible} if for each nonempty $X \in \F$ there is an element $x \in X$ such that $X\sm\{x\} \in \F$.
\item $\F$ is \defterm{intersection-closed} if whenever $X, Y \in \F$, then $X \cap Y \in \F$.
\item $\F$ is \defterm{union-closed} if whenever $X, Y \in \F$, then $X \cup Y \in \F$.
\end{itemize}
Let $\Acc$, $\Int$, and $\Union$ denote respectively  the sets of accessible, intersection-closed, and union-closed set families.
It is routine to check that each of these sets is closed under join, restriction, and contraction, hence is a Hopf submonoid of $\SFam$.  Consequently, antimatroids (accessible set families that are union-closed) form a Hopf submonoid of $\SFam$, as do finite \defterm{near-topologies}, or set families that are intersection- and union-closed.  (To explain this terminology, if $(\mathcal{F},I)$ is a near-topology, then $\mathcal{F}$ contains a unique maximal element $I'\subseteq I$, and is a topology on $I'$; the elements of $I\sm I'$ are phantoms.)  Several more well-known families of combinatorial objects form Hopf submonoids, as we now describe in more detail.

\subsection{Simplicial complexes}

A \defterm{simplicial complex} on $I$ is a grounded set family $(\F,I)$ that is closed under taking subsets.  Since we allow phantom vertices, we do \textit{not} insist that each element of the ground set is actually a vertex of $\F$.  The join of two simplicial complexes is a simplicial complex, and for every decomposition $I=S\sqcup T$, both $\F|_S$ and $\F/_S=\F|_T$ are simplicial complexes.  Therefore, the species $\Simp$ of simplicial complexes forms a cocommutative Hopf submonoid of $\SFam$.  In fact, $\Simp$ is the \textit{universal} cocommutative Hopf submonoid, for the following reason.  Let $\rmH\subset\SFam$ be a cocommutative submonoid, $\F \in \rmH[I]$, $X \in \F$, and $Y \subseteq X$.  Then $Y \in \F|_Y = \F/_{\Bar{Y}}$, which implies that $Y \in \F$.  Hence $\F$ is a simplicial complex, and it follows that $\rmH \subseteq \Simp$.

\subsection{Matroids}

A \defterm{matroid independence complex}, or just a \defterm{matroid complex}, is a simplicial complex $(\F,I)$ that satisfies the ``donation axiom'' or ``augmentation axiom'':
\begin{equation} \label{donation}
\text{if $A,B\in\F$ and $|A|<|B|$, then there exists $x\in B\sm A$ such that $A\cup\{x\}\in\F$.}
\end{equation}
Join and restriction correspond to the elementary matroid operations of direct sum and restriction, respectively \cite[pp.16,22]{Oxley}.  Our operation of contraction does \textit{not} coincide with contraction in the usual matroid sense \cite[p.104]{Oxley}, but rather to restriction to the complement.  Thus the species $\Mat$ of matroids defines a cocommutative submonoid of $\SFam$ (in fact, of $\Simp$) that differs from the (non-cocommutative) Hopf monoid of matroids described in~\cite[\S3.3]{AA}.

\subsection{Lattices of order ideals}

Recall that an \defterm{order ideal} of a poset $P$ is a subset $A\subseteq P$ such that if $x\in A$ and $y<_Px$, then $y\in A$.  Birkhoff's well-known theorem states that the set $J(P)$ of order ideals is a distributive lattice, with meet and join given by intersection and union respectively, and that the correspondence between finite posets and finite distributive lattices is bijective.  Accordingly, we define a set species $\Loi$ by
\[\Loi[I] = \{\text{lattices of order ideals on posets $P$ with ground set contained in $I$}\}.\]
Note the use of ``contained in'' rather than ``equal to'', in order to allow for the possibility of phantoms.  Observe that every $J(P)\in\Loi[I]$ is grounded, since $\0$ is an order ideal of every poset.  (In fact, the set families $J(P)$ are precisely what are known as \emph{poset antimatroids}; see \cite[\S2.3]{KLSgreedoids}, \cite[\S8.7.C]{greedoids}.)

For a poset $P$ and $A\subseteq P$, denote by $P[A]$ the restriction of $P$ to $A$, i.e., the poset on $A$ with order relation inherited from $P$.  It is elementary to check that
\begin{equation} \label{restrictIdeals}
J(P[A]) = J(P)|_A
\end{equation}
where the bar denotes restriction in the sense of~\eqref{restrict-contract}.

\begin{proposition}\label{LOI-Form-Submonoid}
The species $\Loi$ forms a Hopf submonoid of $\SFam$.
\end{proposition}

\begin{proof}
First, $J(P)*J(Q)=J(P+Q)$, where $+$ means disjoint union of posets, so $\Loi$ is closed under join.

Second, we check closure under restriction and contraction.  Let $I=S\sqcup T$ and let $P$ be a poset with underlying set $I'\subseteq I$.  By~\eqref{restrictIdeals} we have $J(P)|_S=J(P[S\cap I'])\in\Loi[S]$ and
\[J(P)/_S = \{A\in J(P) \colon A\subseteq T\} = \{A\in J(P) \colon A\cap S = \0\}.\]
We claim that $J(P)/_S=J(P[U])$, where $U = \{x \in T\colon x \not > y\ \ \forall y \in S\} = I\sm\filter{S}$.  Indeed, if $A\in J(P)$ and $A\cap S=\0$, then in fact $A\subseteq U$
(for if $x\in A$ and $x\notin U$, then there is some $y\in S$ such that $y\leq x$, so $y\in A\cap S$).  It follows that $J(P)/_S\subseteq J(P[U])$.  Conversely, for $B \in J(P[U])$, consider the order ideal $A=\ideal{B}$.  If $A \cap (T\sm U)$ contains an element $x$, then (since $x\not\in U$) there is some $y \in S$ such that $y < x$, but then $y \in A$ and hence $A\cap S\neq\0$, which contradicts the definition of $U$.  We conclude that $A \cap (T\sm U) = \0$, i.e., $B = A$, and it follows that $J(P[U]) \subseteq J(P)/_S$.  So equality holds and we have shown that $\Loi$ is closed under contraction.
\end{proof}

A poset is a \defterm{chain forest} if it is the disjoint union of chains.  The subspecies $\Cf\subset\Loi$ of lattices of order ideals of chain forests (with possible phantoms) is a Hopf monoid.
(More generally, the argument of Proposition~\ref{LOI-Form-Submonoid} implies that any family of posets that is closed under disjoint union, induced subposet, and deletion of order filters gives rise to a Hopf submonoid of $\Loi$.)

A hierarchy of all the Hopf submonoids of $\SFam$ we have described (as well as $\AMat$, $\Acc$, $\Union$, and $\Bool$, all defined in the figure legend) is shown in Figure~\ref{fig:hierarchy}.  We note several observations and questions.
\begin{itemize}
\item We know that $\AMat = \Acc\cap\Union$.  On the other hand, $\Simp\subsetneq\Int\cap\Acc$ and $\Loi\subsetneq\Int\cap\Acc$.  In fact, $\Loi\cup\Simp\subsetneq\Int\cap\Acc$ (although the former is not a Hopf monoid): for instance, the set family $\{\0,1,2,3,12,13,123\}$ is intersection-closed and accessible, but is neither a simplicial complex nor a lattice of order ideals.  We do not know whether $\Int\cap\Acc$ has a nice intrinsic description.  
\item It is a standard fact that a distributive lattice is atomic if and only if it is Boolean \cite[pp.254--255]{ec-1}, from which it follows that $\Simp\cap\Loi=\Bool$.
\item Antimatroids are a special case of \textit{greedoids} \cite{greedoids,KLSgreedoids}: accessible set families that satisfy the donation condition~\eqref{donation}.
The species of greedoids is closed under contraction but not restriction, hence does not form a Hopf submonoid of $\SFam$.  For example, $\F=\{\0,2,4,14,24,23\}$ is a greedoid, but $\F|_{\{1,2,3\}}=\{\0,1,2,23\}$ fails the donation condition with $A=1$, $B=23$.
\end{itemize}

\begin{center}
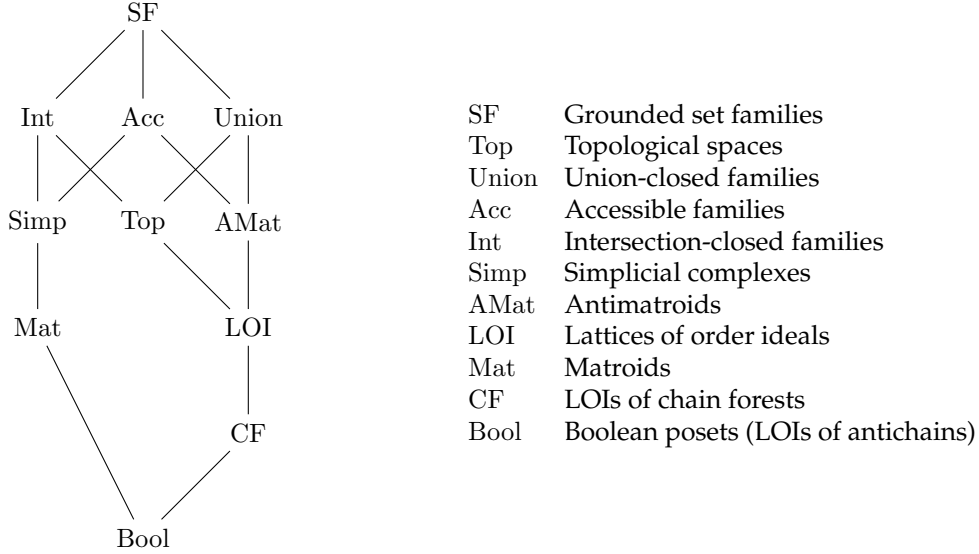
\begin{figure}[htb]
\begin{tikzpicture}[scale=.7]
\node (SFam) at (0,0) {$\SFam$};
\node (NTop) at (0,-4) {$\NTop$};
\node (Union) at (2,-2) {$\Union$};
\node (Acc) at (0,-2) {$\Acc$};
\node (Int) at (-2,-2) {$\Int$};
\node (Simp) at (-2,-4) {$\Simp$};
\node (AMat) at (2,-4) {$\AMat$};
\node (LOI) at (2,-6) {$\Loi$};
\node (CG) at (2,-8) {$\Cf$};
\node (Mat) at (-2,-6) {$\Mat$};
\node (Bool) at (0,-10) {$\Bool$};
\draw (SFam) -- (Union) -- (AMat) -- (LOI);
\draw (SFam) -- (Acc) -- (AMat);
\draw (Acc) -- (Simp) -- (Mat);
\draw (SFam) -- (Int) -- (Simp);
\draw (Int) -- (NTop);
\draw (Union) -- (NTop) -- (LOI);
\draw (Mat) -- (Bool) -- (CG) -- (LOI);
\node at (11,-5) {\begin{tabular}{ll}
$\SFam$ & Grounded set families\\
$\Int$ & Intersection-closed families\\
$\Acc$ & Accessible families\\
$\Union$ & Union-closed families\\
$\Simp$ & Simplicial complexes\\
$\NTop$ & Near-topologies\\
$\AMat$ & Antimatroids\\
$\Mat$ & Matroids\\
$\Loi$ & Lattices of order ideals\\
$\Cf$ & LOIs of chain forests\\
$\Bool$ & Boolean posets (LOIs of antichains)\end{tabular}};
\end{tikzpicture}
\caption{Submonoids of $\SFam$\label{fig:hierarchy}}
\end{figure}
\end{center}

\section{Lattices of order ideals: antipode formula} \label{sec:antipode-LOI}

We now consider the antipode problem for the linearizations of all the Hopf monoids shown in Figure~\ref{fig:hierarchy}.
Obtaining a cancellation-free formula for the antipode in $\SFAM$ appears to be out of reach.  On the other hand, we can use the topological approach of Aguiar and Ardila~\cite{AA} to obtain a concise cancellation-free formula for the antipode in the submonoid $\LOI$, and in particular to show that all the coefficients are $\pm1$.  
By Proposition~\ref{prop:exorcism}, it suffices to consider phantom-free set families of the form $\anti(J(P),I)$, where $P$ is a poset whose underlying set is $I$ (rather than merely a subset of $I$).  Accordingly, we will often suppress the ground set in equations involving the antipode (although one should keep in mind that $\anti(J(P),I)$ may contain terms with phantoms).  The following definition will be crucial in describing precisely which phantoms appear.

\begin{definition} \label{defn:betrayal}
Let $P$ be a poset on $I$, let $\Phi\compn I$, and let $x,y\in I$.  We say $x$ is \defterm{betrayed} by $y$ (with respect to~$P$ and~$\Phi$) if $y<_Px$ and $y<_\Phi x$.  We write $B(\Phi_i)$ for the set of betrayed elements in~$\Phi_i$, and put $B(\Phi)=\bigcup_i B(\Phi_i)$.  Evidently $B(\Phi)\cap\Min(P)=\0$, where $\Min(P)$ denotes the set of minimal elements of $P$.
\end{definition}

\begin{proposition} \label{LOIcprod}
Let $P$ be a poset with ground set $I$, and let $\Phi = \Phi_1|\dots|\Phi_m \compn I$.  Then
\[
\Delta_\Phi(J(P),I) = \bigotimes_{i=1}^m (J(P_i),\Phi_i)
\]
where $P_i=P[\Phi_i\setminus B(\Phi_i)]$.
\end{proposition}

\begin{proof}
Recall from~\eqref{iterated-coproduct} that the $i^{th}$ tensor factor is the set family
\[\F_i = \{A \cap \Phi_i\colon A \in J(P) \text{ and } A\cap(\Phi_1\cup\dots\cup \Phi_{i-1}) = \0\}.\]

If $A \in \F_i$, then $A = \tilde{A}\cap \Phi_i$ for some $\tilde{A} \in J(P)$ and $\tilde{A} \cap(\Phi_1 \cup \dots \cup \Phi_{i-1}) = \0$.  Therefore $x \in A$ implies $y \not< x$ for $y \in \Phi_1 \cup \dots \cup \Phi_{i-1}$, for otherwise $\tilde{A} \cap(\Phi_1 \cup \dots \cup \Phi_{i-1}) \neq \0$ since $y\in \tilde{A} \cap(\Phi_1 \cup \dots \cup \Phi_{i-1})$.  Thus $x \in P_i$ and consequently, $A \in J(P_i)$.

Conversely, for $A = \lfloor x_1, \dots, x_m \rfloor_{P_i}$, let $\tilde{A} = \lfloor x_1, \dots, x_m\rfloor_P$.  Then $A = \tilde{A} \cap \Phi_i$ by~\eqref{restrictIdeals}, and the generators $x_j$ all belong to $P_i$, so $\tilde{A} \cap(\Phi_1 \cup \dots \cup \Phi_{i-1}) = \0$.  Thus $A \in \F_i$.
\end{proof}

Applying $\mu_\Phi$ to the formula of Proposition~\ref{LOIcprod}, we obtain
\begin{equation} \label{muDeltaJP}
\mu_\Phi(\Delta_\Phi(J(P))) = J(P_1)*\cdots*J(P_m) = J(P_1 + \cdots + P_m).
\end{equation}
(Here the operator $*$ indicates the join of the $J(P_i)$ as set families, which is equivalent to the product of the $J(P_i)$ as lattices.)  Equation~\eqref{muDeltaJP} asserts that every term in the antipode has the form $J(Q)$, where $Q=P_1+\cdots+P_m$ and $P_i=P[\Phi_i\sm B(\Phi_i)]$ for some $\Phi\compn I$.  In particular, each component of $Q$ is contained in some block of $\Phi$.  This observation motivates the following definition.

\begin{definition} \label{fracturing}
Let $P$ be a poset with underlying set $I$.  A \defterm{fracturing} $Q$ of $P$ is a disjoint sum of induced subposets of $P$.  (We require only that $Q\subseteq P$ as sets, not that $Q=P$.)  The \defterm{support system} of a fracturing $Q$ is
\[\Supp(Q)=\Supp_P(Q)=\{\Phi\compn I\colon \mu_\Phi(\Delta_\Phi(J(P)))=J(Q)\}\]
and its \defterm{support fan} is
\[\gr{\Supp(Q)}=\{\gr{\Phi}\colon\Phi\in \Supp(Q)\}.\]
A fracturing $Q$ is \defterm{good} if $\Supp(Q)\neq\0$; the set of all good fracturings of $P$ will be denoted $\Good(P)$.  The previous discussion implies that the braid fan is the disjoint union of the support fans $\gr{\Supp(Q)}$ for $Q\in\Good(P)$.
Moreover, Proposition~\ref{LOIcprod} implies that
\begin{equation} \label{who-is-betrayed}
\Phi\in \Supp(Q) \quad\implies\quad B(\Phi)=P\setminus Q \ \text{(as sets)}.
\end{equation}
In particular, every good fracturing must contain $\Min(P)$ as a subset.
\end{definition}

\begin{example} \label{ex:XQ}
Consider the poset $P$ on $I=\{1,2,3\}$ with relations $1<3,2<3$.
The poset $Q$ on $\{1,2\}$ with no relations is a good fracturing of $P$, with support system $\Supp(Q) = \{1|3|2,1|23,1|2|3,12|3,2|1|3,2|31,2|3|1\}$.  The corresponding subfan of the braid fan is shown shaded in Figure~\ref{fig:good-frac-ex}.  (See Section~\ref{sec:posets}.) This example illustrates that a good fracturing~$Q$ need not have the same underlying set as $P$, and that support fans need not be convex.

\begin{figure}[H]
\begin{center}
\scalebox{0.6}{\begin{tikzpicture}
\newcommand{\linlength}{5}
\newcommand{\raypos}{1.8}
\newcommand{\regpos}{3}
\fill[fill=gray!20] (60:\linlength)--(0:0)--(300:\linlength)--(5,-4.3301)--(5,4.3301);
\fill[fill=gray!20] (60:\linlength)--(0:0)--(120:\linlength);
\fill[fill=gray!20] (240:\linlength)--(0:0)--(300:\linlength);
\draw[ultra thick, gray] (0:\linlength)--(0:0);
\draw[ultra thick, gray] (60:\linlength)--(60:0);
\draw[ultra thick, gray] (300:\linlength)--(300:0);
\draw[ultra thick, gray, dashed] (240:\linlength)--(240:0);
\draw[ultra thick, gray, dashed] (180:\linlength)--(180:0);
\draw[ultra thick, gray, dashed] (120:\linlength)--(120:0);
\foreach \a/\lab in {0/$12|3$, 60/$1|23$, 120/$13|2$, 180/$3|12$, 240/$23|1$, 300/$2|31$}
	\node[draw,fill=white] at (\a:\raypos) {\lab};
\foreach \a/\lab in {30/$1|2|3$, 90/$1|3|2$, 150/$3|1|2$, 210/$3|2|1$, 270/$2|3|1$, 330/$2|1|3$}
	\node at (\a:\regpos) {\lab};
\draw[very thick,fill=white] (0,0) circle (.1);
\node at (150:.6) {123};
\node at (0,-\linlength) {\LARGE(b)};
\end{tikzpicture}
} 
\caption{A non-convex support fan.\label{fig:good-frac-ex}}
\end{center}
\end{figure}
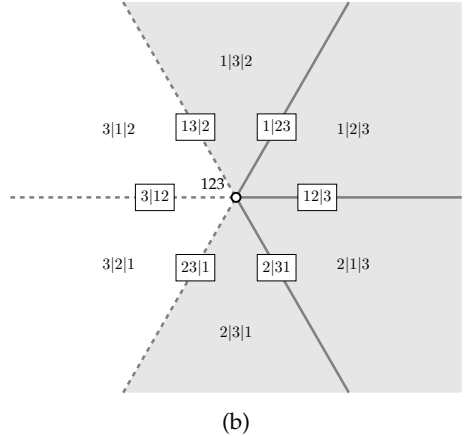
\end{example}

Observe that
\begin{equation} \label{antipode-XQ}
\anti(J(P)) ~=~ \sum_{\Phi\compn I} (-1)^{|\Phi|} \mu_\Phi \Delta_\Phi J(P)
~=~ \sum_{Q\in\Good(P)} J(Q) \underbrace{\left(\sum_{\Phi\in \Supp(Q)} (-1)^{|\Phi|} \right)}_{\varepsilon_Q}.
\end{equation}
Thus, as in \cite[\S1.6]{AA}, the coefficients $\varepsilon_Q$ can be interpreted as relative Euler characteristics of support fans.  The geometry is more complicated than the situation of \cite{AA}, since these fans are not always convex.

\begin{proposition}\label{ethos}
Let $P$ be a poset on ground set $I$, and let $Q$ be a good fracturing of $P$.  Then $\Phi \in \Supp(Q)$ if and only if the following conditions all hold:
\begin{align}
\forall i,j\in Q:\ \ & j <_Q i\ \implies\ i =_\Phi j; \label{stick-together}\\
\forall i,j\in Q:\ \ & j <_P i \text{ and } j \not<_Q i \ \implies\ i <_\Phi j; \label{no-stabby:1}\\
\forall i\in Q:\ \forall j\in P\sm Q:\ \ & j <_P i \ \implies\ i <_\Phi j; \label{no-stabby:2}\\
\forall b \in P\sm Q:\ \exists a\in P:\ \ & a <_P b \text{ and } a <_\Phi b. \label{stabby}
\end{align}
\end{proposition}

\begin{proof}
Suppose $\Phi\in \Supp(Q)$.  By Proposition~\ref{LOIcprod}, each component of $Q$ is contained in some block of $\Phi$, implying~\eqref{stick-together}.  By \eqref{who-is-betrayed}, $B(\Phi)\supseteq P\sm Q$, which is equivalent to~\eqref{stabby}; and $B(\Phi)\subseteq P\sm Q$, which implies~\eqref{no-stabby:1} and~\eqref{no-stabby:2}.

Conversely, suppose that $\Phi\compn I$ satisfies \eqref{stick-together}--\eqref{stabby}.  Let $Q'$ be the good fracturing of $P$ such that $\Phi\in \Supp(Q')$.  
First, we claim that $B(\Phi)=P\sm Q$.  The inclusion $B(\Phi)\supseteq P\sm Q$ is just~\eqref{stabby}.  For the reverse inclusion, if $i\in B(\Phi)\cap Q$, then there exists $j\in P$ such that $j<_Pi$ and $j<_\Phi i$.  If $j\in P\sm Q$ then~\eqref{no-stabby:2} fails, while if $j\in Q$ then~\eqref{no-stabby:1} implies $j<_Qi$, but then $i=_\Phi j$ by~\eqref{stick-together}, a contradiction, so the claim is proved.  In particular, $Q'=Q$ as sets.  By~\eqref{stick-together}, every relation in $Q$ is a relation in $Q'$; conversely, if $j<_{Q'}i$, then $j<_Pi$ and $j=_\Phi i$, so~\eqref{no-stabby:1} implies $j<_Qi$.  Therefore, $Q=Q'$.
\end{proof}

By condition~\eqref{stick-together}, $\gr{\Supp(Q)}$ is contained in the subspace $V_Q\subset\Rr^n$ defined by equalities $x_i=x_j$ whenever $i,j$ belong to the same component of $Q$.  We have $\dim V_Q=u+k$, where $u$ is the number of components of~$Q$ and $k=|P\sm Q|$.  Condition~\eqref{stabby} gives rise to a disjunction of linear inequalities rather than a conjunction, which is why $\gr{\Supp(Q)}$ need not be convex (see Example~\ref{ex:XQ}).  Accordingly, our next step is to express $\gr{\Supp(Q)}$ as a union of convex fans.

\begin{definition}\label{betrayal-function}
Let $Q$ be a fracturing of $P$.  A \defterm{betrayal function} is a map $\beta:P\sm Q\to P$ such that $\beta(b) <_P b$ for every $b\in P\sm Q$. Observe that $Q$ has a betrayal function if and only if $Q\supseteq\Min(P)$.  Let
\[\Supp_\beta(Q)=\{\Phi\in \Supp(Q):\ \beta(b) <_\Phi b \ \ \forall b\in P\sm Q\}.\]
Applying Proposition~\ref{ethos}, we  see that $\Supp_\beta(Q)$ consists of set compositions $\Phi\in \Supp(Q)$ satisfying~\eqref{stick-together}, \eqref{no-stabby:1}, \eqref{no-stabby:2}, and
\begin{equation} \label{XbetaQ}
\forall b\in P\sm Q:\ \ \beta(b) <_\Phi b.
\end{equation}
\end{definition}

Observe that $\gr{\Supp_\beta(Q)}$ is a \emph{convex} subfan of the braid arrangement for every $\beta$, because for each $b\in P\sm Q$, the disjunction~\eqref{stabby} has been replaced by a single inequality.  Moreover, $\Supp(Q)=\bigcup_\beta \Supp_\beta(Q)$, though in general this is not a disjoint union.

\begin{proposition} \label{topology-of-Xbeta}
Let $Q$ be a good fracturing of $P$ and $\beta$ a betrayal function for $Q$.  Then $\gr{\Supp_\beta(Q)}$ is homeomorphic to $\Rr^{u+k}$, where $u$ is the number of components of $Q$ and $k=|P\sm Q|$.
\end{proposition}
\begin{proof}
The affine hull of $\gr{\Supp_\beta(Q)}$ is defined by the linear equalities \eqref{stick-together}, hence has one degree of freedom for each component of $Q$ and each element of $P\sm Q$.  The inequalities given by \eqref{no-stabby:1}, \eqref{no-stabby:2}, and~\eqref{XbetaQ} define $\gr{\Supp_\beta(Q)}$ as a convex open subset of its affine hull.  The conclusion follows by \cite{Geschke}.
\end{proof}

\begin{example} \label{ex:XbQ}
Recall the poset $P$ and good fracturing $Q$ of Example~\ref{ex:XQ}.  In Figure~\ref{fig:landing-in-oz}, each face of $\Supp(Q)$ is colored \Green{green}, \Blue{blue}, or \Red{red}, depending on whether $3\in P$ is betrayed by \Green{only by 1}, \Blue{only by 2}, or \Red{by both 1 and 2}.  There are two betrayal functions $\beta_1,\beta_2:\{3\}\to\{1,2\}$, given by $\beta_i(3) = i$.  Thus $\gr{\Supp_{\beta_1}(Q)}$ is the subfan consisting of the green and red faces, and $\gr{\Supp_{\beta_2}(Q)}$ consists of the blue and red faces.  Observe that both subfans are convex.
 
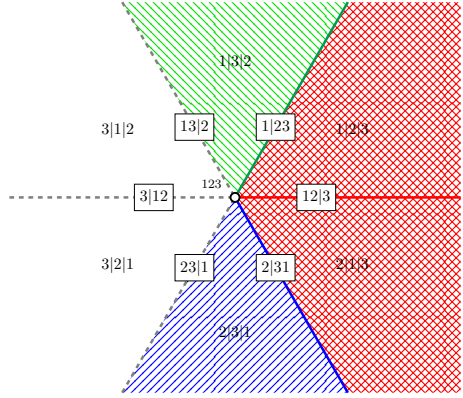
\begin{figure}[H]
\centering
\scalebox{0.6}{
\begin{tikzpicture}
\newcommand{\linlength}{5}
\newcommand{\raypos}{1.8}
\newcommand{\regpos}{3}
\fill[pattern color=red, pattern=crosshatch] (60:\linlength)--(0:0)--(300:\linlength)--(5,-4.3301)--(5,4.3301);
\fill[pattern color=ForestGreen, pattern=north west lines] (60:\linlength)--(0:0)--(120:\linlength);
\fill[pattern color=blue, pattern=north east lines] (240:\linlength)--(0:0)--(300:\linlength);
\draw[ultra thick,red] (0:\linlength)--(0:0);
\draw[ultra thick,ForestGreen] (60:\linlength)--(60:0);
\draw[ultra thick,blue] (300:\linlength)--(300:0);
\draw[ultra thick, gray, dashed] (240:\linlength)--(240:0);
\draw[ultra thick, gray, dashed] (180:\linlength)--(180:0);
\draw[ultra thick, gray, dashed] (120:\linlength)--(120:0);
\foreach \a/\lab in {0/$12|3$, 60/$1|23$, 120/$13|2$, 180/$3|12$, 240/$23|1$, 300/$2|31$}
	\node[draw,fill=white] at (\a:\raypos) {\lab};
\foreach \a/\lab in {30/$1|2|3$, 90/$1|3|2$, 150/$3|1|2$, 210/$3|2|1$, 270/$2|3|1$, 330/$2|1|3$}
	\node at (\a:\regpos) {\lab};
\draw[very thick,fill=white] (0,0) circle (.1);
\node at (150:.6) {\footnotesize$123$};
\end{tikzpicture}}
\caption{A geometric realization of a support system.\label{fig:landing-in-oz}}
\end{figure}
\end{example}

In order to understand the support of $\anti(P)$, we need to know which fracturings are good.  The following definition and proposition give a usable criterion for goodness, together with a way of constructing explicit elements of $\Supp(Q)$ for a fracturing $Q$.

\begin{definition}\label{def:conflict-digraph}
Let $Q$ be a fracturing of $P$ with (Hasse) components $Q_1,\dots,Q_u$.  The \defterm{conflict digraph} $\Con_P(Q)$ of $Q$ is the directed graph with vertices $Q_1,\dots,Q_u$ and edges
\[
\{Q_i \to Q_j\colon i\neq j;\ \exists\; x\in Q_i,\ y\in Q_j \text{  such that } y<_P x\}.
\]
\end{definition}

\begin{example}
Consider the poset $P=\{1<2<3,4<5<6,1<4,2<5,3<6\}$ and the fracturing whose components are the induced subposets $Q_1 = \{4\}$, $Q_2 = \{1<2<5\}$, $Q_3= \{3<6\}$.  Then $\Con_P(Q)$ is the digraph with vertices $Q_i$ and edges $Q_1 \to Q_3$, $Q_2\to Q_1$, $Q_3 \to Q_2$, and $Q_3 \to Q_1$.  (See Figure~\ref{fig:frac-ex}.)

\begin{figure}[H]
\centering
\scalebox{0.8}{
\begin{tikzpicture}
\draw [thick] (0,0) -- (2,2) -- (3,1) -- (1,-1) -- (0,0);
\draw [thick] (1,1) -- (2,0);

\node[fill=red!20] at (1,-1) {1};
\node[fill=red!20] at (2, 0) {2};
\node[fill=blue!20] at (3, 1) {3};
\node[fill=ForestGreen!20] at (0, 0) {4};
\node[fill=red!20] at (1, 1) {5};
\node[fill=blue!20] at (2, 2) {6};
\draw[thick, <-] (5.5,-1.15) -- (8.5,-1.15);
\draw[thick, ->] (5.5,-.85) -- (8.5,-.85);
\draw[thick, ->] (6.7,0.65) -- (5.2,-0.65);
\draw[thick, ->] (7.3,0.65) -- (8.8,-0.65);
\node[fill=ForestGreen!20] at (5, -1) {$Q_1$};
\node[fill=red!20] at (9,-1) {$Q_2$};
\node[fill=blue!20] at (7,1) {$Q_3$};
\end{tikzpicture}}
\caption{$P$ and $\Con_P(Q)$.\label{fig:frac-ex}}
\end{figure}

Observe that $Q$ cannot be a good fracturing.  For any set composition $\Phi\compn[6]$ satisfying~\eqref{stick-together}, either 4 betrays 5, or 1 betrays 4, or $\{1,2,4,5\}$ is contained in some block of $\Phi$, in which case $\mu_\Phi(\Delta_\Phi(P))$ includes the relations $1<4<5$.  This obstruction to goodness is captured by the antiparallel edges between $Q_1$ and $Q_2$ in $\Con_P(Q)$.
\end{example}

In fact, cycles in $\Con_P(Q)$ form obstructions to goodness, as we now explain.
  
\begin{definition}\label{def:acyclic-fracturing}
Let $P$ be a poset.  An \defterm{acyclic fracturing} $Q$ of $P$ is a fracturing of $P$ such that the conflict digraph $\Con_P(Q)$ is acyclic.  Let $\Acyc(P)$ denote the set of all acyclic fracturings of $P$.
\end{definition}

Equivalently, a fracturing is acyclic if its Hasse components can be ordered $Q_1,\dots,Q_u$ such that $Q_i$ contains an element less than an element of $Q_j$ only if $i<j$.

\begin{proposition} \label{classify-good-fracturings}
Let $Q$ be a fracturing of $P$ with components $Q_1,\dots,Q_u$. Then $Q$ is a good fracturing if and only if $Q\supseteq\Min(P)$ and $Q$ is an acyclic fracturing.
\end{proposition}

\begin{proof}
($\implies$) Suppose that $Q$ is a good fracturing, and let $\Phi\in \Supp(Q)$.  We have observed in Definition~\ref{fracturing} that $Q\supset\Min(P)$.  Now, suppose that $\Con_P(Q)$ contains a cycle, which we may take to be $Q_1\rightarrow\cdots\rightarrow Q_s\rightarrow Q_1$.  Then there are elements $x_1,\dots,x_s,y_1,\dots,y_s$ of $Q$, with $x_j,y_j\in Q_j$ and $x_j<_Py_{j+1}$ for all $j$ (taking indices mod~$s$).  For each $j$, since $x_j$ does not betray $y_{j+1}$, it follows that $Q_j \subseteq \Phi_{r_j}$ and $Q_{j+1} \subseteq \Phi_{r_{j+1}}$, where $r_{j+1} \leq r_j \leq m$.  But then $r_1 \geq r_2 \geq \dots \geq r_m \geq r_1$, 
so all the $r_i$ are equal (say to $r$) and $Q_1\cup\cdots\cup Q_s\subset\Phi_r$.  By Proposition~\ref{LOIcprod}, $x_j <_Q y_{j+1}$ for all $j$, but then the $Q_i$ are in fact identical.  So the cycle is a self-loop, which is prohibited in the construction of $\Con_P(Q)$.  We conclude that $\Con_P(Q)$ is acyclic and thus $Q$ is an acyclic fracturing.
\medskip

($\impliedby$) Recall the definitions and properties of preposets from \S\ref{sec:posets}.

Suppose that $Q\supset\Min(P)$ and that $\Con_P(Q)$ is acyclic. As observed in Definition~\ref{betrayal-function}, the first assumption implies that $Q$ admits a betrayal function $\beta$.  Acyclicity of $\Con_P(Q)$ implies that there is a well-defined preposet $O_\beta(Q)$ on $I$ with equivalence classes $\{Q_1,\dots,Q_u\}\cup (P\sm Q)$, generated by the relations
\begin{subequations}
\begin{align}
Q_j < b & \quad\text{whenever } \beta(b)\in Q_j; \label{conflict-leaf}\\
Q_i < Q_j & \quad\text{whenever } Q_i \rightarrow Q_j \text{ is an edge in }\Con_P(Q).\label{conflict-core}
\end{align}
\end{subequations}
Let $\Phi$ be a linear extension of $O_\beta(Q)$.  That is, $\Phi$ is a linear preposet with the same equivalence classes that contains the relations ~\eqref{conflict-leaf} and~\eqref{conflict-core}.  
The conditions~\eqref{conflict-leaf} imply that $P\sm Q\subseteq B(\Phi)$, while the conditions~\eqref{conflict-core} imply the reverse inclusion. Moreover, the posets $P_i$ of Proposition~\ref{LOIcprod} are precisely the components of $Q$.  Therefore $\mu_\Phi(\Delta_\Phi(P))=Q$, i.e., $\Phi\in \Supp(Q)$, and now~\eqref{conflict-leaf} implies that in fact
$\Phi \in \Supp_\beta(Q)$.
\end{proof}

\begin{remark}\label{convexity}
Each component $Q'$ of a good fracturing $Q$ of $P$ has a notable property: if $x,y\in Q'$, $z\in P$, and $x\leq_P z\leq_P y$, then $z\in Q'$.
(In poset terminology, $Q'$ is interval-closed as a subposet of $(Q,\leq_P)$.)
Indeed, let $\beta$ be a betrayal function and $\Phi$ a linear extension of $O_\beta(Q)$.  If $z \not\in B(\Phi)$, then let $\Phi_i$ be the block of $\Phi$ containing $z$ and $\Phi_j$ be the block of $\Phi$ containing $x$ and $y$.  It is not the case that $i<j$ (when $z$ betrays $y$) or that $i>j$ (when $x$ betrays $z$), so $i=j$, and it follows that $z \in Q_j$.
\end{remark}

\begin{proposition}\label{intersection-happiness}
Let $P$ be a poset, let $Q\in\Good(P)$, and let $\mathcal{B}$ be the collection of all betrayal functions $\beta$ such that $\Supp_\beta(Q)\neq\0$. Then $\underset{\beta \in \mathcal{B}}{\bigcap} \Supp_\beta(Q) \neq \0$.
\end{proposition}

\begin{proof}
Let $D$ be the digraph $(V_1\cup V_2,E_1\cup E_2\cup E_3)$, where $V_1=\{Q_1,\dots,Q_u\}$ and $V_2=P\sm Q$, and
\begin{align*}
E_1 &= \{Q_i \rightarrow Q_j:\ \text{$i\neq j$ and there exist $x\in Q_i$ and $y\in Q_j$ such that $y<_P x$}\},\\
E_2 &= \{b \rightarrow b':\ b,b'\in P\sm Q \text{ and } b <_P b'\},\\
E_3 &= \{Q_j \rightarrow b:\ b\in P\sm Q \text{ and there exists $x \in Q_j$ such that $x <_P b$}\}.
\end{align*}
It suffices to show that $D$ is acyclic, for then every linear extension of $D$ belongs to $\underset{\beta \in \mathcal{B}}{\bigcap} \Supp_\beta(Q)$.  Indeed, the subdigraph $(V_1,E_1)$ is just the conflict digraph $\Con_P(Q)$, which is acyclic by Proposition~\ref{classify-good-fracturings}, and the subdigraph $(V_2,E_2)$, is the transitive closure of the poset $P\sm Q$.  Moreover, every edge in $E_3$ points from $V_1$ to $V_2$.  Thus $D$ is acyclic as desired.
\end{proof}

\begin{theorem}\label{thm:antipode-LOI}
Let $P$ be a finite poset on ground set~$I$.  Then the antipode of $J(P)$ in $\LOI[I]$ is given by the following cancellation-free and grouping-free formula:
\[
\anti(J(P),I) = \sum_{Q\in\Good(P)} (-1)^{c(Q)+|P\sm Q|} (J(Q),I)
\]
where $c(Q)$ is the number of components of $Q$.
\end{theorem}

\begin{proof}
Recall from~\eqref{antipode-XQ} that
\[\anti(J(P)) ~=~ \sum_{Q\in\Good(P)} \varepsilon_Q J(Q)\]
where
\[\varepsilon_Q=\sum_{\Phi\in \Supp(Q)} (-1)^{|\Phi|}.\]
Fix a good fracturing $Q$, and abbreviate $u=c(Q)$ and $k=|P\sm Q|$.  Recall that $\Supp(Q) = \bigcup\limits_{\beta \in \mathcal{B}} \Supp_\beta(Q)$, where $\mathcal{B}$ is the set of all betrayal functions for $Q$.  For $\mathcal{A}\subseteq\mathcal{B}$, define
\[Y_\mathcal{A}=Y_\mathcal{A,Q}=\bigcap\limits_{\beta\in\mathcal{A}}\Supp_\beta(Q).\]
By inclusion-exclusion, we have
\begin{equation} \label{cQ:1}
\varepsilon_Q = \sum_{\0\neq\mathcal{A}\subseteq\mathcal{B}}(-1)^{|\mathcal{A}|+1} \sum_{\Phi\in Y_\mathcal{A}} (-1)^{|\Phi|}.
\end{equation}
As in the proof of \cite[Thm.~1.6.1]{AA}, the inner sum can be interpreted as the reduced Euler characteristic of $\gr{Y_\mathcal{A}}$ as a relative polyhedral complex (or equivalently of $\gr{Y_\mathcal{A}}\cap\Ss^{n-2}$ as a relative simplicial complex; see \S\ref{sec:posets}).
Since each $\gr{\Supp_\beta(Q)}$ is open, convex, and homeomorphic to $\Rr^{u+k}$ (by Proposition~\ref{topology-of-Xbeta}), so is their intersection $\gr{Y_\mathcal{A}}$.  Thus $\gr{Y_\mathcal{A}}\cap\Ss^{n-2}$ is homeomorphic to an open ball of dimension $u+k-2$, and thus
\[\sum_{\Phi\in Y_\mathcal{A}} (-1)^{|\Phi|}
= \tilde\chi(\overline{Y_\mathcal{A}}) - \tilde\chi(\partial Y_\mathcal{A})
= \tilde\chi(\mathbb{B}^{u+k})-\tilde\chi(\mathbb{S}^{u+k-1})
= 0 - (-1)^{u+k-1}
= (-1)^{u+k}.
\]
Substituting into \eqref{cQ:1}, we obtain
\[\varepsilon_Q
= \sum_{\0\neq\mathcal{A}\subseteq\mathcal{B}}(-1)^{|\mathcal{A}|+1+u+k}\\
= (-1)^{u+k+1}\sum_{\0\neq\mathcal{A}\subseteq\mathcal{B}}(-1)^{|\mathcal{A}|}\\
= (-1)^{u+k}
\]
which establishes the desired formula for $\anti(J(P))$.
\end{proof}

\begin{corollary}\label{cor:dual-poset-antipode}
Let $P$ be a poset and let $P^*$ denote the dual of $P$.  Then
\[
\anti(J(P^*)) = \sum_{\substack{Q \in \Acyc(P) \\ \Max(P) \subseteq Q}} (-1)^{c(Q) + |P\sm Q|} J(Q).
\]
\end{corollary}

\begin{proof}
By Proposition~\ref{classify-good-fracturings} we see that $Q^*$ is a good fracturing of $P^*$ if and only if $Q^*\in\Acyc(P^*)$ and  $\Min(P^*)\subseteq Q^*$.  By taking the dual of each component of $Q^*$ we obtain a fracturing $Q$ of $P$.  $Q$ is an acyclic fracturing (of $P$) since dualizing each component will reverse the edges of the conflict graph of $Q^*$.  It also follows that $\Min(P^*) \subseteq Q$ and since $\Min(P^*) = \Max(P)$ we can write $\Max(P) \subseteq Q$.  Since $c(Q) = c(Q^*)$ and $|P\sm Q| = |P^*\sm Q^*|$ we can rewrite the antipode formula from Theorem~\ref{thm:antipode-LOI} for $J(P^*)$ as
\[
\anti(J(P^*)) = \sum_{\substack{Q \in \Acyc(P) \\ \Max(P) \subseteq Q}} (-1)^{c(Q) + |P\sm Q|} J(Q).
\]
\end{proof}

\begin{remark}
The duality map $\varphi:\LOI\to\LOI$ given by $\varphi(J(P))=J(P^*)$ is not a Hopf automorphism or antiautomorphism: it respects join and restriction, but not contraction.  For example, consider the zigzag poset $Z$ (Fig.~\ref{fig:zigzag}).  If $A = \{2\}$, then $\varphi(J(Z)/_A) = J(\{1\})$ but $\varphi(J(Z))/_A = J(\{3<1,4\})$.
\end{remark}

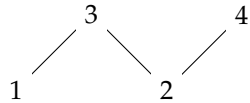
\begin{figure}[ht]
    \centering
    \begin{tikzpicture}
        \draw (0,0)--(1,1)--(2,0)--(3,1);
        \node[fill=white] at (0,0) {1};
        \node[fill=white] at (2,0) {2};
        \node[fill=white] at (1,1) {3};
        \node[fill=white] at (3,1) {4};
    \end{tikzpicture}
    \caption{The zigzag poset $Z$.\label{fig:zigzag}}
\end{figure}

\begin{example} \label{exa:antipode-for-antichains}
Let $A_n$ be an antichain with $n$ elements.  Then $\Min(A_n)=A_n$, so the only good fracturing of $A_n$ is $A_n$ itself; the components are its singleton subsets.  We have $u=n$ and $k=0$, so
\[
\anti(J(A_n)) = (-1)^n J(A_n).
\]
Moreover, $J(A_n)$ is just the Boolean algebra with atoms $A_n$, so this formula describes the antipode in the submonoid of Boolean set families (with no phantoms).
\end{example}

\begin{example} \label{exa:antipode-for-chains}
Let $P=C_n$ be the naturally ordered chain on $I=[n]$, so that $J(C_n)$ is the chain $\0 \subset [1] \subset [2] \subset \dots \subset [n]$.  Observe that $\Min(C_n)=\{1\}$, so the good fracturings are disjoint sums of chains $Q=Q_1+\cdots+Q_u$ such that $1\in Q_1$ and the blocks of $Q$ give a natural set composition $\Psi$ of a subset $V\subseteq[n]$.  (Here ``natural'' means that if $i<j$, then every element of $Q_i$ is strictly less than every element of $Q_j$.)
Therefore,
\begin{equation} \label{eq:antipode-for-chains}
\anti(J(C_n)) = \sum_{\substack{V\subseteq[n]:\\ 1\in V}} (-1)^{n-|V|} \sum_{\substack{\Psi=\Psi_1|\cdots|\Psi_u\compn V\\ \text{natural}}} (-1)^u J(C_{\Psi_1})*\cdots*J(C_{\Psi_u}).
\end{equation}
\end{example}

\subsection{Ordinal sums} \label{subsec:antipode-ordinal-sums}
Recall that the \defterm{ordinal sum} $\Pl \oplus \Ph$ is constructed from $\Pl+\Ph$ by adding the relations $x < y$ for all $x \in \Pl$ and $y \in \Ph$.  The good fracturings of $\Pl\oplus\Ph$ can be classified in terms of fracturings of $\Pl$ and $\Ph$, enabling us to give a formula for the antipode in an ordinal sum.

\begin{definition}
\label{def:pure-and-mixed}
Let $Q$ be a good fracturing of $P = \Pl \oplus \Ph$.  First, we say that $Q$ is \defterm{pure} if $Q = \Ql + \Qh$, where $\Ql$ is a good fracturing of $\Pl$ and $\Qh$ is an acyclic fracturing of $\Ph$.  (That is, every component of $Q$ is a subposet either of $\Pl$ or of $\Ph$.)  Second, we say that $Q$ is \defterm{mixed} if it has a component $H$ such that $H \cap \Pl \neq \0$ and $H \cap \Ph \neq \0$.  In this case $H$ is a \defterm{hybrid component}.  Note that $Q$ can have at most one hybrid component, since any two such would form a 2-cycle in the conflict digraph, so we may write $Q=\Ql+\Qh+H$, where $\Ql$ (resp., $\Qh$) is the subposet consisting of components contained in $\Pl$ (resp., $\Qh$).
\end{definition}

We will treat pure and mixed fracturings separately.  The summands in Theorem~\ref{thm:antipode-LOI} arising from pure fracturings are easier to describe.

\begin{proposition}[Pure fracturings]\label{prop:pure}
Let $P= \Pl \oplus \Ph$ where $\Pl \neq \0$.  Then
\[
\sum_{\text{pure\ }Q} (-1)^{c(Q)+|P\sm Q|}J(Q) = \anti(J(\Pl)) * \sum_{\Qh\in\Acyc(\Ph)} (-1)^{c(\Qh) + |\Ph\sm\Qh|} J(\Qh).
\]
\end{proposition}

\begin{proof}
If $Q$ is pure, then $Q = \Ql + \Qh$ where $\Ql$ is a good fracturing of $\Pl$ and $\Qh$ is an acyclic fracturing of $\Ph$.  The conflict digraph $\Con_P(Q)$ is formed from the acyclic digraph $\Con_P(\Ql)+\Con_P(\Qh)$ by adding edges from every component of $\Qh$ to every component of $\Ql$, but not vice versa (see Figure~\ref{fig:conflicts}), so it too is acyclic.  Moreover, $\Min(P) = \Min(\Pl)$, so $Q$ is in fact a good fracturing of $P$.  Therefore
\begin{align*}
\sum_{\text{pure\ }Q} (-1)^{c(Q)+|P\sm Q|}J(Q)
&= \sum_{\Ql \in \Good(\Pl)}\sum_{\Qh\in\Acyc(\Ph)} (-1)^{c(\Ql + \Qh)+|P\sm(\Ql+\Qh)|}J(\Ql+\Qh)\\
&= \sum_{\Ql \in \Good(\Pl)}\sum_{\Qh\in\Acyc(\Ph)} (-1)^{c(\Ql) + c(\Qh) + |\Pl\sm\Ql| + |\Ph\sm\Qh|}J(\Ql)*J(\Qh)\\
&= \anti(J(\Pl)) * \sum_{\Qh\in\Acyc(\Ph)} (-1)^{c(\Qh) + |\Ph\sm\Qh|} J(\Qh).\qedhere
\end{align*}
\end{proof}

\begin{figure}[htb]
\begin{center}
\begin{tikzpicture}[scale=.7]
\newcommand{\xell}{1.6} \newcommand{\yell}{.7} 
\newcommand{\ycon}{2} 

\node at (0,\ycon) {$\Con_P(\Qh)$};
\node at (0,-\ycon) {$\Con_P(\Ql)$};
\draw (0,\ycon) ellipse ({\xell} and \yell);
\draw (0,-\ycon) ellipse ({\xell} and \yell);
\draw [latex-] (0,1.25*\yell-\ycon)--(0,\ycon-1.25*\yell);
\foreach \x in {-1,1} \draw [latex-] (\x,1.1*\yell-\ycon)--(\x,\ycon-1.1*\yell);

\begin{scope}[shift={(8,0)}]
\node at (0,\ycon) {$\Con_P(\Qh)$};
\node at (0.55,0) {$H$};
\node at (0,-\ycon) {$\Con_P(\Ql)$};
\draw (0,\ycon) ellipse ({\xell} and \yell);
\fill (0,0) circle (1.25mm);
\draw (0,-\ycon) ellipse ({\xell} and \yell);
\draw [latex-] (0,.1)--(0,\ycon-1.25*\yell);
\draw [latex-] (-.15,0)--(-.85,\ycon-1.1*\yell);
\draw [latex-] (.15,0)--(.85,\ycon-1.1*\yell);
\draw [-latex] (-1,\ycon-1.1*\yell)--(-1,1.1*\yell-\ycon);
\draw [-latex] (1,\ycon-1.1*\yell)--(1,1.1*\yell-\ycon);
\draw [-latex] (0,0)--(0,1.25*\yell-\ycon);
\draw [-latex] (0,0)--(-.85,1.1*\yell-\ycon);
\draw [-latex] (0,0)--(.85,1.1*\yell-\ycon);
\end{scope}
\end{tikzpicture}
\caption{Conflict graphs of a pure fracturing (left) and a mixed fracturing (right).\label{fig:conflicts}}
\end{center}
\end{figure}
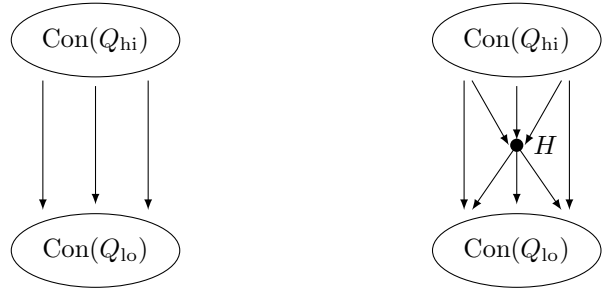

\begin{proposition}[Mixed fracturings] \label{prop:mixed}
Let $P = \Pl \oplus \Ph$, and suppose that $Q=\Ql+\Qh+H$ is a mixed fracturing of $P$, as in Definition~\ref{def:pure-and-mixed}.
Then $Q$ is a good fracturing of $P$ if and only if
\begin{enumerate}
\item the induced subposet $Q[Q\cap\Pl]$ is a good fracturing of $\Pl$;
\item the induced subposet $Q[Q\cap\Ph]$ is an acyclic fracturing of $\Ph$;
\item $H \cap \Pl$ is a order filter in $P[Q\cap \Pl]$;
\item $H \cap \Ph$ is a order ideal in $P[Q \cap \Ph]$.
\end{enumerate}
Consequently,
\[
\sum_{\text{mixeds\ }Q} (-1)^{c(Q)+|P\sm Q|}J(Q)
= \sum_{\substack{H\in\Hyb(\Pl,\Ph)\\ \Ql \in \Good(\Pl\sm \filter{H}) \\ \Qh \in \Acyc(\Ph\sm \ideal{H})}} (-1)^{c(\Ql)+c(\Qh)+1 + |P| - |\Ql| - |\Qh|-|H|} J(\Ql \oplus H \oplus \Qh).
\]
where $\Hyb(\Pl,\Ph)$ denotes the set of induced subposets $H$ of $P$ that intersect both $\Pl$ and $\Ph$ (thus, potential hybrid components of a mixed fracturing).
\end{proposition}

\begin{proof}
First note that $\Min(Q) = \Min(\Pl) \subseteq \Pl \cap Q$.  To show that $Q\cap\Pl$ and $Q\cap\Ph$ are acyclic, observe that $\Con_P(Q[Q\cap\Pl])$ is obtained from $\Con_P(Q)$ by removing all components whose restriction to $\Pl$ is empty.  Thus $\Con_P(Q[Q\cap\Pl])$ is isomorphic to a subgraph of $\Con_P(Q)$ and thus $\Con_P(Q[Q\cap\Pl])$ is acyclic. Likewise, $\Con_P(Q[Q\cap\Ph])$ is isomorphic to a subgraph of $\Con_P(Q)$ and thus $\Con_P(Q[Q\cap\Ph])$ is acyclic.  Hence $Q[Q\cap\Pl]$ is a good fracturing of $\Pl$ and $Q[Q\cap\Ph]$ is an acyclic fracturing of $\Ph$.

To show that $H \cap \Pl$ is an order filter of $Q \cap \Pl$, suppose that $a,b \in \Pl$ such that $a < b$, $a \in H$ and $b$ belongs to some component $Y$ of $Q$.  Then $\Con_P(Q)$ has an edge $Y \to H$.  Since $H \cap \Ph \neq \0$ we also know there is an edge $H \to Y$.  Since $\Con_P(Q)$ is acyclic, it must be the case that $H = Y$ and thus $b \in H \cap \Pl$.  Hence $X \cap \Pl$ is an order filter of $Q \cap \Pl$.  A similar argument shows that $H \cap \Ph$ is an order ideal of $Q \cap \Ph$.  Thus we have verified that conditions (1)--(4) are necessary for $Q$ to be a good fracturing of $P$.
\medskip

Conversely, suppose $Q$ is a fracturing of $P$ with a unique hybrid component $H$ and that conditions (1)--(4) hold.  Since $\Min(P) = \Min(\Pl) \subseteq Q[Q\cap\Pl]\subset Q$, we need only show that $\Con_P(Q)$ is acyclic.

Consider the induced subdigraphs $G_{\mathrm{lo}} = \Con_P(Q[Q\cap(\Ph\sm H)])$ and $G_{\mathrm{hi}} = \Con_P(Q[Q\cap(\Ph\sm H)])$.  As subdigraphs of an acyclic digraph, $G_{\mathrm{lo}}$ and $G_{\mathrm{hi}}$ are acyclic.  Furthermore, we claim that $\Con_P(Q)$ consists of the disjoint union $G_{\mathrm{lo}} + G_{\mathrm{hi}}$ together with the vertex $H$ and the edges
\[\left\{
B_{\mathrm{hi}} \to B_{\mathrm{lo}},\ 
H \to B_{\mathrm{lo}},\ 
B_{\mathrm{hi}} \to H ~\colon~ B_{\mathrm{lo}} \in G_{\mathrm{lo}},\ B_{\mathrm{hi}} \in G_{\mathrm{hi}}
\right\}\]
(see Figure~\ref{fig:conflicts}).
Indeed, the graph thus constructed is a subgraph of $\Con_P(Q)$ since every edge comes from a relation between two components of $Q$.  Since every element of $\Ph$ is greater than every element of $\Pl$, there are no edges of the form $B_{\mathrm{lo}} \to B_{\mathrm{hi}}$.  There are no edges of the form $B_{\mathrm{lo}} \to H$ by assumption (3), and no edges of the form $H\to B_{\mathrm{hi}}$ by assumption (4).  Hence the constructed graph must be $\Con_P(Q)$.  Thus $Q$ is a good fracturing of~$P$.
\end{proof}

Combining Propositions~\ref{prop:pure} and~\ref{prop:mixed} yields a cancellation-free formula for the antipode of the ordinal sum of two posets:

\begin{theorem} \label{thm:antipode-of-ordinal-sum}
Suppose $P = \Pl \oplus \Ph$.  Then
\begin{align*}
\anti(J(P)) &= \anti(J(\Pl)) * \sum_{\Qh\in\Acyc(\Ph)} (-1)^{c(\Qh) +  |\Ph\sm\Qh|} J(\Qh)\\ 
&+ \sum_{\substack{H\in\Hyb(\Pl,\Ph)\\ \Ql \in \Good(\Pl\sm \filter{H}) \\ \Qh \in \Acyc(\Ph\sm \ideal{H})}} (-1)^{c(\Ql)+c(\Qh)+1 + |P| - |\Ql| - |\Qh|-|H|} J(\Ql \oplus H \oplus \Qh).
\end{align*}
\end{theorem}

\begin{example}
When $P$ is an antichain, setting $\Ph=P$ and $\Pl=\0$ recovers Example~\ref{exa:antipode-for-antichains}.

  More generally, suppose that $P$ is a \defterm{complete ranked poset}: that is, $P$ is ranked, and every pair of elements of different ranks are comparable.  Equivalently, $P$ is an ordinal sum of antichains, so Theorem~\ref{thm:antipode-of-ordinal-sum} can be specialized to give an inductive formula for the antipode of $P$.  Specifically, we can write $P=\Pl\oplus\Ph$, where $\Ph$ is the set of maximal elements of $P$ (equivalently, elements of maximal rank), so that $\Acyc(\Ph)=2^{\Ph}$, and Theorem~\ref{thm:antipode-of-ordinal-sum} implies that
\begin{align*}
\anti(J(P)) &= \anti(J(\Pl)) * \sum_{\Qh\subseteq\Ph} (-1)^{|\Qh|} 2^{\Qh} \\ 
&+\sum_{H \in \Hyb(\Pl,\Ph)} \sum_{\Ql \in \Good(\Pl\sm \filter{H})} \sum_{\Qh \subseteq \Ph\sm \ideal{H}} (-1)^{c(\Ql)+1 + |(\Pl\sm H)\sm\Ql| + |\Ph\sm H|} J(\Ql \oplus H \oplus \Qh).
\end{align*}
One can then recursively apply this formula to $\Pl$ to obtain the antipode of the complete ranked poset~$P$.
\end{example}

\section{Simplicial complexes: progress toward an antipode} \label{sec:submonoid-simp}

We now consider the problem of the antipode for the Hopf submonoid $\SIMP$ of simplicial complexes.  The cocommutativity of $\SIMP$ implies that permuting the blocks of a set composition does not change the corresponding summand in Takeuchi's formula.  Consequently, Takeuchi's formula can be rewritten as a sum over set partitions, with many fewer terms; on the other hand, we cannot expect the coefficients to be $\pm1$.
Another convenient consequence of cocommutativity is that contraction does not produce any new phantoms.  On the other hand, there does not appear to be a complete description of support systems as there was for $\LOI$, and so a general cancellation-free antipode formula for $\SIMP$ seems unobtainable.

Throughout this section, let $X$ be a simplicial complex on ground set $[n]$ (with possible phantoms).  For each $\Phi = \Phi_1|\dots|\Phi_k \compn [n]$, define
\begin{equation} \label{prod-cprod-simplicial}
\begin{aligned}
\decomp{X}{\Phi} &= \mu_{\Phi}(\Delta_{\Phi}(X))
= X|_{\Phi_1}*\dots*X|_{\Phi_k}\\
&= \langle \sigma = \sigma_1\cup\dots\cup \sigma_k\colon \sigma_i \text{ is a facet of } X|_{\Phi_i} \rangle
\end{aligned}
\end{equation}
where $\langle F\rangle$ means the simplicial complex generated by the list of faces $F$.
In particular, permuting the blocks of $\Phi$ does not change the subcomplex $\decomp{X}{\Phi}$.  Accordingly, from now on, we will work with set \textit{partitions} $\Phi\partn[n]$ (unordered decompositions of $I$ into blocks) rather than set \textit{compositions} $\Phi\compn[n]$.

If $X=\decomp{X}{\Phi}$, we say that $X|_{\Phi_1}*\dots*X|_{\Phi_k}$ is a \defterm{join decomposition of $X$ of length $k$.}

Recall some standard facts about the partition lattice \cite[Ex.~3.10.4]{ec-1}.  The family $\Pi_n$ of set partitions of $[n]$ is partially ordered by reverse refinement: for $\Phi,\Psi\partn I$, we have $\Phi\leq\Psi$ if every block of $\Psi$ is a union of blocks of $\Phi$.  This poset is in fact a (geometric) lattice: the meet $\Phi\wedge\Psi$ is the coarsest common refinement of $\Phi$ and $\Psi$.

Rewritten in terms of set partitions, Takeuchi's formula~\eqref{Takeuchi} becomes
\begin{equation} \label{Takeuchi-cc}
\anti_I(X) = \sum_{\Phi \partn I} (-1)^{|\Phi|} |\Phi|!\, \decomp{X}{\Phi}.
\end{equation}
Accordingly, in analogy to Definition~\ref{fracturing}, we define the \defterm{support system} of the pair $X,Y$ as
\[\Supp_X(Y) = \{\Phi\in \Pi_n \colon Y=\decomp{X}{\Phi}\}.\]
Observe that if $\Supp_X(Y)\neq\0$ then $X\subseteq Y$.  Accordingly, in this case we say that $Y$ is an \defterm{inflation} of $X$, or that $\Phi\in\Supp_X(Y)$ \defterm{inflates} $X$ into $Y$.

\begin{proposition} \label{refine}
Suppose that $\Phi=\Phi_1|\cdots|\Phi_k$ and $\Psi=\Psi_1|\cdots|\Psi_\ell$ are set partitions such that $\decomp{X}{\Phi}=\decomp{X}{\Psi}$.  Then in fact $\decomp{X}{\Phi}=\decomp{X}{\Psi}=\decomp{X}{\Omega}$, where $\Omega$ is the common refinement $\Phi\wedge\Psi$.  (I.e., the blocks of $\Omega$ are the nonempty sets of the form $\Phi_i\cap\Psi_j$.)
\end{proposition}

\begin{proof}
The hypothesis $\decomp{X}{\Phi}=\decomp{X}{\Psi}$ is equivalent to $X_{\Phi_1}*\dots*X_{\Phi_k}=X_{\Psi_1}*\dots*X_{\Psi_\ell}$, i.e.,
\[\{\sigma_1\cup\dots\cup\sigma_k\colon \sigma_i\in X,\ \sigma_i\subseteq \Phi_i\} =
\{\tau_1\cup\dots\cup\tau_\ell\colon \tau_i\in X,\ \tau_i\subseteq \Psi_i\}.\]
In particular, both $\decomp{X}{\Phi}$ and $\decomp{X}{\Psi}$ are subcomplexes of $\mu_\Omega(\Delta_\Omega(X))$.  On the other hand, every face $\omega\in\Omega$ is a union of faces
\[\{\omega_{ij} \in X_{\Phi_i \cap \Psi_j} \colon i\in[k], \ j\in[\ell]\}.\]
For $i\in[k]$, let $\omega_i = \omega_{i1} \cup \dots \cup \omega_{i\ell}$; then $\omega_i \in X_{\Psi_i}*\dots*X_{\Psi_\ell} = X_{\Phi_i}*\dots*X_{\Phi_k}$, and $\omega_i \subset \Phi_i$, so in fact $\omega_i \in X_{\Phi_i}$.  It follows that $\omega=\omega_1\cup\cdots\omega_k\in\decomp{X}{\Phi}$.  A similar argument shows that $\omega\in\decomp{X}{\Psi}$.
\end{proof}

\begin{corollary} \label{sc-support}
Every nonempty support system $\Supp_X(Y)$ is meet-closed and interval-closed.
\end{corollary}
\begin{proof}
Closure under meet is just Proposition~\ref{refine}.  Moreover, if $\Phi<\Psi<\Omega$ are partitions, then in general
$\decomp{X}{\Phi}\supseteq\decomp{X}{\Psi}\supseteq\decomp{X}{\Omega}$, so if $\Supp_X(Y)$ contains both $\Phi$ and $\Omega$ then it contains $\Psi$ as well.
\end{proof}

In particular, each nonempty $\Supp_X(Y)$ contains a unique finest partition $\FInf_X(Y)$, which we will call the \defterm{fundamental inflator} of $X$ with respect to~$Y$.  The collection of all fundamental inflators of $X$ will be denoted $\Fund(X)$.  We may further group Takeuchi's formula~\eqref{Takeuchi-cc} as
\begin{equation} \label{Takeuchi-grouped}
\anti_I(X) = \sum_{\Phi\in\Fund(X)} c_\Phi \decomp{X}{\Phi},
\quad\text{where}\quad
c_\Phi=\sum_{\Psi\partn I:\ \decomp{X}{\Phi}=\decomp{X}{\Psi}} (-1)^{|\Psi|} |\Psi|!.
\end{equation}
This expression is cancellation-free in the sense that the basis elements $\decomp{X}{\Phi}$ are all distinct, although the coefficients are not necessarily in the simplest possible form.

\begin{proposition}
The coefficient of $X$ in $\anti(X)$ is $(-1)^{m(X)}$, where $m(X)$ is the maximum length of a join decomposition of $X$.
\end{proposition}
\begin{proof}
The support system $\Supp_X(X)$ is always nonempty, since it contains a unique maximal element, namely the partition of $[n]$ with one block.  In particular, $\Supp_X(X)$ is a sublattice of $\Pi_n$, and its minimal element $\FInf_X(X)$ describes the finest way to represent $X$ as a join of smaller complexes, namely $m(X)$ of them.  Thus $\Supp_X(X)\isom\Pi_{m(X)}$ as lattices.  Using the set-composition version of Takeuchi's formula, the coefficient of $X$ in $\anti(X)$ is just the Euler characteristic of the triangulation of $\Ss^{m(X)-2}$ by the braid arrangement, namely $(-1)^{m(X)}$.
\end{proof}

\begin{example}
Let $X=\langle\{i\}:\ i\in[n]\rangle$ be the 0-dimensional complex with vertices $[n]$.  Then for each $\Phi\in\Pi_n$, the complex $\decomp{X}{\Phi}$ is the join of the 0-dimensional complexes on the blocks of $\Phi$.  In particular, the support systems are all singletons.  Each $\Phi$ arises from $|\Phi|!$ set compositions, so
\[\anti(A) = \sum_{\Phi\in\Pi_n} (-1)^{|\Phi|}|\Phi|! \decomp{X}{\Phi}.\]
\end{example}

\begin{example}
Let $X=\langle 123,34\rangle$.  Then the partition version of Takeuchi's formula ~\eqref{Takeuchi-grouped} gives
\[\anti(X) = X
+4\Red{\overbrace{\langle1234\rangle}^{X_1}}
-2\Blue{\overbrace{\langle123,234\rangle}^{X_2}}
-2\Green{\overbrace{\langle123,134\rangle}^{X_3}}.
\]
Figure~\ref{fig:suppsysX} shows the support systems corresponding to each simplicial complex occurring in the antipode.  As proved in Corollary~\ref{sc-support}, each support system is meet-closed and interval-closed, with a unique minimum (the fundamental inflator).

\begin{figure}[ht]
\begin{center}
\begin{tikzpicture}
\newcommand{\xsc}{2.2}
\newcommand{\ysc}{1.5}
\newcommand{\delt}{1.3}
\coordinate (top) at (0,3*\ysc);
\coordinate (c123) at (0*\xsc,2*\ysc);
\coordinate (c124) at (1*\xsc,2*\ysc);
\coordinate (c13) at (-3*\xsc,2*\ysc);
\coordinate (c234) at (-2*\xsc,2*\ysc);
\coordinate (c12) at (-1*\xsc,2*\ysc);
\coordinate (c134) at (2*\xsc,2*\ysc);
\coordinate (c23) at (3*\xsc,2*\ysc);
\coordinate (a12) at (1.5*\xsc,\ysc);
\coordinate (a13) at (-1.5*\xsc,\ysc);
\coordinate (a14) at (2.5*\xsc,\ysc);
\coordinate (a23) at (0.5*\xsc,\ysc);
\coordinate (a24) at (-2.5*\xsc,\ysc);
\coordinate (a34) at (-0.5*\xsc,\ysc);
\coordinate (bot) at (0,0);
\foreach \coatom in {c123,c124,c134,c234,c12,c13,c23} \draw (top)--(\coatom);
\foreach \coatom in {c123,c124,c12} \draw (a12)--(\coatom);
\foreach \coatom in {c123,c134,c13} \draw (a13)--(\coatom);
\foreach \coatom in {c124,c134,c23} \draw (a14)--(\coatom);
\foreach \coatom in {c123,c234,c23} \draw (a23)--(\coatom);
\foreach \coatom in {c124,c234,c13} \draw (a24)--(\coatom);
\foreach \coatom in {c134,c234,c12} \draw (a34)--(\coatom);
\foreach \atom in {a12,a13,a14,a23,a24,a34} \draw (bot)--(\atom);
\node[fill=white] at (a12) {\Red{$12|3|4$}};
\node[fill=white] at (a13) {\Red{$13|2|4$}};
\node[draw=black,thick,fill=white] at (a14) {\Blue{$14|2|3$}};
\node[fill=white] at (a23) {\Red{$23|1|4$}};
\node[draw=black,thick,fill=white] at (a24) {\Green{$24|1|3$}};
\node[fill=white] at (a34) {\Red{$34|1|2$}};
\node[fill=white] at (c123) {\Red{$123|4$}};
\node[draw=black,thick,fill=white] at (c124) {{$124|3$}};
\node[fill=white] at (c134) {\Blue{$134|2$}};
\node[fill=white] at (c234) {\Green{$234|1$}};
\node[fill=white] at (c12) {\Red{$12|34$}};
\node[fill=white] at (c13) {\Green{$13|24$}};
\node[fill=white] at (c23) {\Blue{$14|23$}};
\node[draw=black,thick,fill=white] at (bot) {\Red{$1|2|3|4$}};
\node[fill=white] at (top) {{$1234$}};

\draw[very thick,rounded corners=1cm] (-2.5*\xsc,\ysc-\delt) -- (-3*\xsc-\delt,2*\ysc+\delt/2) -- (-2*\xsc+\delt,2*\ysc+\delt/2) -- cycle;
	\node at (-3.5*\xsc,\ysc) {\Green{$X_3$}};
\draw[very thick,rounded corners=1cm] (2.5*\xsc,\ysc-\delt) -- (3*\xsc+\delt,2*\ysc+\delt/2) -- (2*\xsc-\delt,2*\ysc+\delt/2) -- cycle;
	\node at (3.5*\xsc,\ysc) {\Blue{$X_2$}};
\draw[very thick,rounded corners=.7cm] (0,-\delt/2) -- (-1.4*\xsc-\delt,\ysc) -- (-1*\xsc,2*\ysc+\delt/2) -- (0*\xsc-\delt/2,2*\ysc+.6*\delt) -- (.75*\xsc,\ysc+.75*\delt) -- (1.5*\xsc+.9*\delt,\ysc) -- cycle;
	\node at (\xsc+\delt/2,0) {\Red{$X_1$}};
\draw[very thick,rounded corners=.5cm] (0,3*\ysc+\delt/2) -- (\xsc+\delt,2*\ysc) -- (\xsc,2*\ysc-\delt/2)-- (-\delt,3*\ysc) -- cycle;
	\node at (-\delt,3*\ysc+\delt/2) {$X$};
\end{tikzpicture}
\end{center}
\caption{The lattice $\Pi_4$, stratified into support systems of the complex $X=\langle123,34\rangle$.  Fundamental inflators are \fbox{boxed.}\label{fig:suppsysX}}
\end{figure}
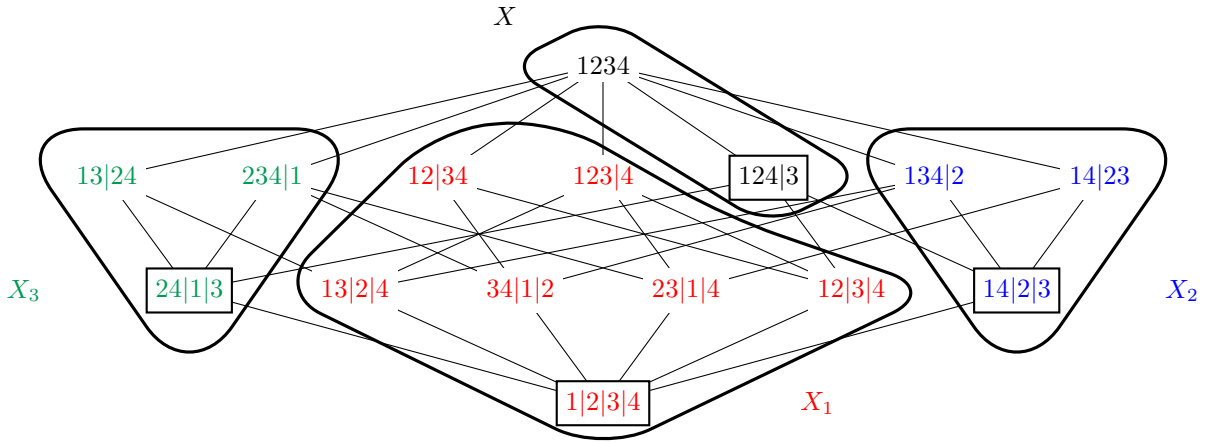
\end{example}

\begin{proposition} \label{prop:fun-canon}
Suppose $X$ and $Y$ are simplicial complexes, and let $\Psi = \Psi_1 | \dots |\Psi_k \in \Supp_X(Y)$ .
Then $\Psi$ is the fundamental inflator of $Y$ if and only if $Y = X_{\Psi_1}*\cdots*X_{\Psi_k}$ is the unique finest join decomposition of $Y$.
\end{proposition}

\begin{proof}
Suppose $\Psi = \Psi_1|\dots|\Psi_k = \FInf_X(Y)$ and that the canonical decomposition of $Y$ is $X_{\Phi_1} * \dots * X_{\Phi_\ell}$ for $\Phi = \Phi_1|\dots|\Phi_\ell \in \Pi_n$.  Certainly $Y = X_{\Psi_1} * \dots * X_{\Psi_k}$ is a join decomposition of $Y$, so $\Psi\geq\Phi$.  On the other hand,
\[X_{\Phi_1}*\cdots*X_{\Phi_\ell}
\supseteq X_{\Psi_1}*\cdots*X_{\Psi_k} = Y
=Y_{\Phi_1}*\cdots*Y_{\Phi_\ell}
\supseteq X_{\Phi_1}*\cdots*X_{\Phi_\ell}
\]
(since $Y\supseteq X$), so equality holds throughout.  In particular $\Phi\in\Supp_X(Y)$, so $\Phi\geq\Psi$ and equality holds.
\end{proof}

The problem of calculating antipode coefficients in general appears to be intractable, for the following reason.  Consider a support system $\Supp_X(Y)$ with minimal element $\Phi=\FInf_X(Y)$ and maximal elements $\Omega_1,\dots,\Omega_k$.  For $A\subseteq[k]$, let $\Omega_A=\bigwedge_{a\in A}\Omega_a$; then by inclusion/exclusion
\begin{align*}
\sum_{\Psi\in\Supp_X(Y)} (-1)^{|\Psi|} |\Psi|!
&= \sum_{\0\neq A\subseteq[k]} (-1)^{|A|-1} \sum_{\Psi\in\bigcap_{a\in A}[\Phi,\Omega_a]} (-1)^{|\Psi|} |\Psi|!\\
&= \sum_{\0\neq A\subseteq[k]} (-1)^{|A|-1} \sum_{\Psi\in[\Phi,\Omega_A]} (-1)^{|\Psi|} |\Psi|!.
\end{align*}
Each interval $[\Phi,\Omega_A]$ is a sublattice of $\Pi_n$, hence a product of smaller partition lattices.  Specifically, if the $b$ blocks of $\Omega_A$ are respectively broken into $k_1,\dots,k_b$ blocks in $\Phi$, then $[\Phi,\Omega_A]\isom\Pi_{k_1}\x\cdots\x\Pi_{k_b}$.  In the special case $[\Phi,\Omega_A]\isom\Pi_k$, we have in fact
\begin{align*}
\sum_{\Psi\in[\Phi,\Omega_A]} (-1)^{|\Psi|} |\Psi|!
&= (-1)^{|\Phi|} \sum_{\Theta\in\Pi_k} (-1)^{k-|\Theta|} (|\Theta|+b-1)!\\
&= (-1)^{|\Phi|} \sum_{j=1}^k S(k,j) (-1)^{k-j} (j+b-1)!\\
&= (-1)^{|\Phi|} (b-1)! \sum_{j=1}^k S(k,j) (-1)^{k-j} b(b+1)(b+2)\cdots(b+j-1)\\
&= (-1)^{|\Phi|} (b-1)! b^k
\end{align*}
by \cite[p.249, identity 6.12]{GKP}.
Unfortunately, if $[\Phi,\Omega_A]$ is a nontrivial product of partition lattices, then there does not appear to be a simple formula for 
$\sum_{\Psi\in[\Phi,\Omega_A]} (-1)^{|\Psi|} |\Psi|!$.

\subsection{Join closure of simplex skeletons}
For a set of vertices $V$, the \defterm{simplex skeleton} $\sk(m,V)$ is the simplicial complex whose faces are the subsets of $V$ of cardinality at most $m$ (i.e., dimension $m-1$).  We abbreviate $\sk(m,n)=\sk(m,[n])$.  Observe that $\sk(m,n)$ is join indecomposable whenever $m<n$.  Moreover,
\begin{equation}\label{restrict-skeleton}
\sk(m,n)|_X = \sk(\min(|X|,m),X) \qquad \forall X\subseteq[n].
\end{equation}

\begin{proposition}\label{finf-skeletons}
A partition $\Phi\in\Pi_n$ is a fundamental inflator of $\sk(m,n)$ if and only if every block $\Phi_i\in\Phi$ satisfies $|\Phi_i|=1$ or $|\Phi_i| >m$.
\end{proposition}
\begin{proof}
By~\eqref{restrict-skeleton}, this condition is equivalent to saying that every subcomplex $\sk(m,n)|_{\Phi_i}$ is join-irreducible.
\end{proof}

Let $p_{a,b}(c)$ denote the number of partitions of a $c$-element set into $b$ blocks, each of cardinality at most $a$.
Moreover, for a set partition $\Phi$, let $s(\Phi)$ and $t(\Phi)$ denote respectively the number of singleton and non-singleton blocks of $\Phi$.  
  
\begin{theorem} \label{thm:simplex-skeleton}
Let $X = \sk(m,n)$.  Then
\[
\anti(X) = \sum_{\Phi \in \Fund(X)} \left(\sum_{j=0}^{s(\Phi)} (-1)^{t(\Phi)+j} p_{m,j}(s(\Phi))  (t(\Phi)+j)!\right) \decomp{X}{\Phi}.
\]
\end{theorem}

\begin{proof}
Let $\Phi\in\Fund(X)$ and $Y=\decomp{X}{\Phi}$.  As a consequence of Proposition~\ref{finf-skeletons}, the partitions $\Theta\in\Supp_X(Y)$ are precisely those obtained by merging singleton blocks of $\Phi$; however, no more than $m$ singletons of $\Phi$ can be merged to form a block of $\Theta$ without changing $\decomp{X}{\Theta}$.  The number of such mergers $\Theta$ in which exactly $j$ blocks of $\Theta$ are unions of singletons of $\Phi$ is $p_{m,j}(s(\Phi))$.  Each such partition has $t(\Phi)+j$ blocks, hence gives rise to $(t(\Phi)+j)!$ set compositions.  The formula now follows from the set-partition version~\eqref{Takeuchi-grouped} of Takeuchi's formula.
\end{proof}

\begin{remark}
Let $I=I_1\sqcup\cdots\sqcup I_k$. The corresponding \defterm{complete colorful complex} is
\[X=X(I_1,\dots,I_k)=\{\sigma\subseteq I:\ |\sigma\cap I_j|\leq 1\ \forall j\}.\]
We regard each set $I_j$ as consisting of vertices of ``color'' $j$; the faces of~$X$ are then the sets of vertices with no more than one of any color.  In particular $\FInf_X(X)=I_1|\cdots|I_k$, and the other fundamental inflators are precisely the further refinements of this partition, i.e., the partitions in which every block is monochromatic.  To get the bottom element of the support system containing a given partition $\Theta$, break up each block of $\Theta$ into its maximal monochromatic subsets.  An example with $I_1=\{1,2\}$, $I_2=\{3,4\}$ is shown in Figure~\ref{fig:ccc}.  Meanwhile, a partition $\Phi$ is maximal in its support system if and only if for every pair of blocks, there is some color represented in both blocks (so that any further coarsening will add faces to $\decomp{X}{\Phi}$).  On the other hand, it is not clear how to systematically compute the coefficient of the antipode corresponding to a support system, or how to effectively list the maximal elements of a support system (in order to carry out inclusion/exclusion).
\end{remark}

\begin{figure}[th]
\begin{center}
\begin{tikzpicture}
\newcommand{\xsc}{2.2}
\newcommand{\ysc}{1.5}
\newcommand{\delt}{1.3}
\coordinate (top) at (0,3*\ysc);
\coordinate (c12) at (0*\xsc,2*\ysc);
\coordinate (c124) at (3*\xsc,2*\ysc);
\coordinate (c13) at (-1*\xsc,2*\ysc);
\coordinate (c234) at (-2*\xsc,2*\ysc);
\coordinate (c123) at (2*\xsc,2*\ysc);
\coordinate (c134) at (-3*\xsc,2*\ysc);
\coordinate (c23) at (1*\xsc,2*\ysc);
\coordinate (a12) at (2.5*\xsc,\ysc);
\coordinate (a13) at (-1.5*\xsc,\ysc);
\coordinate (a14) at (0.5*\xsc,\ysc);
\coordinate (a23) at (1.5*\xsc,\ysc);
\coordinate (a24) at (-0.5*\xsc,\ysc);
\coordinate (a34) at (-2.5*\xsc,\ysc);
\coordinate (bot) at (0,0);
\foreach \coatom in {c123,c124,c134,c234,c12,c13,c23} \draw (top)--(\coatom);
\foreach \coatom in {c123,c124,c12} \draw (a12)--(\coatom);
\foreach \coatom in {c123,c134,c13} \draw (a13)--(\coatom);
\foreach \coatom in {c124,c134,c23} \draw (a14)--(\coatom);
\foreach \coatom in {c123,c234,c23} \draw (a23)--(\coatom);
\foreach \coatom in {c124,c234,c13} \draw (a24)--(\coatom);
\foreach \coatom in {c134,c234,c12} \draw (a34)--(\coatom);
\foreach \atom in {a12,a13,a14,a23,a24,a34} \draw (bot)--(\atom);

\node[draw=black, thick, fill=white] at (c12) {{$12|34$}};
\node[fill=white] at (top) {{$1234$}};

\node[draw=black, thick, fill=white] at (a12) {\Red{$12|3|4$}};
\node[fill=white] at (c123) {\Red{$123|4$}};
\node[fill=white] at (c124) {\Red{$124|3$}};

\node[draw=black, thick, fill=white] at (a34) {\Green{$34|1|2$}};
\node[fill=white] at (c134) {\Green{$134|2$}};
\node[fill=white] at (c234) {\Green{$234|1$}};

\node[draw=black, thick, fill=white] at (bot) {\Blue{$1|2|3|4$}};

\node[fill=white] at (a13) {\Blue{$13|2|4$}};
\node[fill=white] at (a14) {\Blue{$14|2|3$}};
\node[fill=white] at (a23) {\Blue{$23|1|4$}};
\node[fill=white] at (a24) {\Blue{$24|1|3$}};
\node[fill=white] at (c13) {\Blue{$13|24$}};
\node[fill=white] at (c23) {\Blue{$14|23$}};

\end{tikzpicture}
\caption{Support systems in a complete colorful complex.\label{fig:ccc}}
\end{center}
\end{figure}
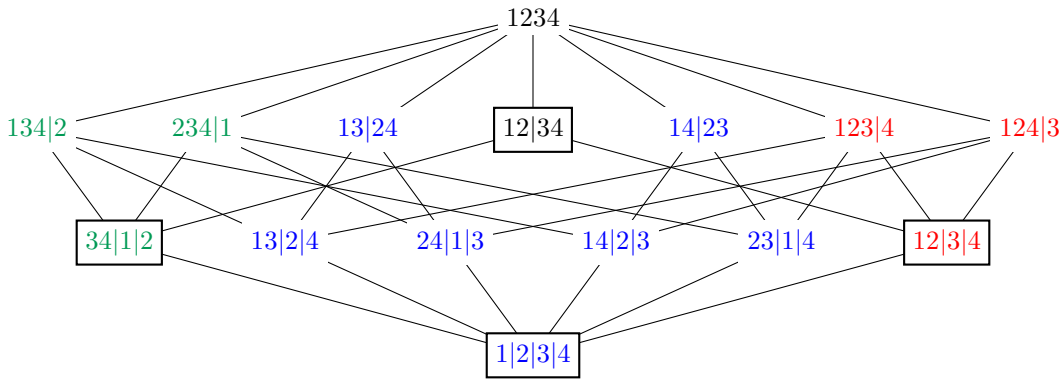

\section{Chain forests: symmetric functions and characters} \label{sec:chainforest}

A \defterm{chain forest} is a poset $P$ that is the disjoint union of chains.  Every chain forest can be specified up to isomorphism by a (possibly empty) integer partition $\lambda$ whose parts are the sizes of its maximal chains.  The lattices of order ideals of chain forests (with possible phantoms) generate a subspecies $\CF\subset\LOI$:
\[\CF[I] = \Bbbk\{J(P)\colon P\text{ is a chain forest with ground set } J\subseteq I\}.\]
In fact $\CF$ is a Hopf submonoid of $\LOI$, since chain forests with phantoms are closed under disjoint unions and induced subposets.  A chain forest with all chains of size~1 is just an antichain, so $\CF$ contains the Hopf monoid $\BOOL$ generated by Boolean lattices with possible phantoms.

Let $\pha$ be the trivial set family with one phantom, and let $\chain{n}$ be a complete flag on an $n$-element set, with no phantoms; in particular, $\chain{0}$ is the multiplicative unit.
Let $\cg{\lambda}{p}$ denote the equivalence class of lattices of order ideals of chain forests with chain sizes $\lambda$ and $p$ phantoms.  
These equivalence classes span the Hopf algebra $\CFA$ obtained from $\CF$ by applying the Fock functor $\fock$; specifically, the $k$th graded piece of $\CFA$ has basis $\{\cg{\lambda}{p}\colon |\lambda|+p=k\}$.  (For details on $\fock$, see \cite[\S1.1.10]{AA} and \cite{AguiarMahajan};  for Hopf algebras in a combinatorial context, see, e.g., \cite{BenSag} or \cite{ABS}.)  The Hopf algebra $\CFA$ will be the focus of this section.

\subsection{The structure of \texorpdfstring{$\CFA$}{the chain forest Hopf algebra}}

As a graded algebra, $\CFA$ is the free polynomial algebra generated by the phantom $\pha$ and the chains $\chain{n}$ for $n\geq1$.  That is, the product is given by
\begin{equation}\label{eq:cg-alg-prod}
(\cg{\lambda}{p})(\cg{\mu}{q})=\cg{\lambda\cup\mu}{p+q}
\end{equation}
where $\lambda\cup\mu$ is the multiset union of $\lambda$ and $\mu$, sorted in decreasing order.

The coproduct $\Delta$ is a morphism of algebras, i.e.,
\[\Delta(\cg{(\lambda_1,\dots,\lambda_\ell)}{p})=\Delta(\chain{\lambda_1})\cdots\Delta(\chain{\lambda_\ell})\Delta(\pha)^p,\]
so it suffices to compute
\begin{align}
\Delta(\pha) &= 1\otimes\pha + \pha\otimes 1, \label{eq:cg-alg-cprod-phantom} \\
\Delta(\chain{n})
&= \sum_{I \subseteq [n]} [\chain{n} |_I]\otimes [\chain{n}/_I]\notag\\
&= 1\otimes\chain{n} + \sum_{\0\neq I\subseteq[n]} \chain{|I|} \otimes \cg{\min(I)-1}{n-|I|-\min(I)+1}\notag\\
&= 1\otimes\chain{n} + \sum_{m=1}^n \sum_{J\subseteq[m+1,n]} \chain{|J|+1} \otimes \cg{m-1}{n-m-|J|}\notag\\
\intertext{(where $m=\min(I)$ and $J=I\sm\{m\}$)}
&= 1\otimes\chain{n} + \sum_{m=1}^n \sum_{j=0}^{n-m} \binom{n-m}{j} \chain{j+1} \otimes \cg{m-1}{n-m-j}.\label{eq:cg-alg-cprod-chain}
\end{align}
In particular, the expression~\eqref{eq:cg-alg-cprod-chain} is phantom-free if and only if $n\leq2$.

Our main structural result on $\CFA$ is that the Hopf algebra of symmetric functions arises as a quotient.  We start by recalling some of the basic theory of symmetric functions; a standard reference is Chapter~7 of \cite{ec-2}.
Let $h_n$ and $e_n$ denote respectively the complete homogeneous and elementary symmetric functions of degree~$n$ in commuting indeterminates $x_1,x_2,\dots$ over a field $\Bbbk$.  The algebra of symmetric functions $\Lambda=\Lambda_\Bbbk$ is the free polynomial algebra $\Bbbk[h_1,h_2,\dots]=\Bbbk[e_1,e_2,\dots]$.  In fact, it is a Hopf algebra \cite[Chapter~2]{GR}, with coproduct given in these bases by
\begin{equation} \label{Lambda-coproduct}
\Delta(h_n) = \sum_{i+j=n} h_i \otimes h_j, \qquad \Delta(e_n) = \sum_{i+j=n} e_i \otimes e_j
\end{equation}
and antipode
\[\anti(h_n)=(-1)^n e_n.\]
The generating functions
\[
H(t) = \sum_{k\geq0} h_k t^k = \prod_{i\geq 1} \frac{1}{1-tx_i},\qquad
E(t) = \sum_{k\geq0} e_k t^k = \prod_{i\geq 1} (1+tx_i).
\]
satisfy $H(t)E(-t)=1$.  This power series equation can be used to solve for the $h_n$ and $e_n$ in terms of each other; the resulting determinant formula is a special case of the Jacobi-Trudi determinant formula for Schur functions.

Recall \cite[chapter~1]{montgomery} that a \emph{Hopf ideal} $\calI$ of a Hopf algebra $\calH$ is a $\Bbbk$-vector subspace $\calI$ that is (a) an ideal ($\calH\calI\subseteq \calI$), (b) a coideal ($\Delta(\calI)\subseteq \calI\otimes \calH + \calH\otimes \calI$), and (c) closed under the antipode.  (The third condition follows from the first two in a graded connected Hopf algebra, where Takeuchi's formula holds.)  The quotient $\calH/\calI$ is then a vector space that inherits a Hopf algebra structure from $\calH$.

\begin{theorem}\label{thm:punchline}
Let $\calI$ be the vector space spanned by all elements $\cg{\lambda}{p}$, where $\lambda$ is a partition and $p > 0$.  Then $\calI$ is a Hopf ideal and the map $f:\Lambda\to\CFA/\calI$ sending $h_\lambda$ to $\overline{\chain{\lambda}}$ (i.e. the image of $\chain{\lambda}\in\CFA$ modulo~$\calI$) is an isomorphism of Hopf algebras.
\end{theorem}

\begin{proof}
Evidently $\calI$ is an ideal of $\CFA$, and it is a coideal by~\eqref{eq:cg-alg-cprod-phantom}, hence a Hopf ideal. 
Thus the $n$th graded piece of $\CFA/\calI$ is the vector space generated by $\{\overline{\chain{\lambda}}\colon\lambda\partn n\}$.  Consider the vector space isomorphism $f:\Lambda \to \CFA/\calI$ defined by $f(h_\lambda) = \overline{\chain{\lambda}}$.  It is in fact a ring isomorphism because $f(h_\lambda h_\mu) = f(h_{\lambda\cup\mu}) = \chain{\lambda}\cup\chain{\mu}=\chain{\lambda}\chain{\mu}=f(h_\lambda)f(h_\mu)$.  Meanwhile, the coproduct in $\CFA/\calI$ is obtained by setting $\pha=0$ in~\eqref{eq:cg-alg-cprod-chain}, or equivalently extracting the terms with $j=m-n$, namely
\[
\Delta(\chain{n}) = 1\otimes\chain{n} + \sum_{m=1}^n \chain{n-m+1} \otimes \chain{m-1} = \sum_{k=0}^n \chain{k}\otimes\chain{n-k}
\]
which corresponds to the coproduct formula~\eqref{Lambda-coproduct} in~$\Lambda$.  Thus $f$ is an isomorphism of graded connected bialgebras, hence of Hopf algebras.
\end{proof}

Of course, $h_\lambda$ could be replaced with $e_\lambda$ throughout Theorem~\ref{thm:punchline}.  While it is tempting to identify $\Lambda$ with the ``phantom-free'' vector subspace $\CFA_0\subseteq\CFA$ spanned by the elements $\chain{\lambda}=\cg{\lambda}{0}$, this identification is an isomorphism only of algebras, not of coalgebras.  Indeed, $\CFA_0$ is not itself a coalgebra: by the remark after equation~\eqref{eq:cg-alg-cprod-chain}, if $\lambda$ has a part of size~2 or greater, then $\Delta(\chain{\lambda})\notin\CFA_0$.

\subsection{Characters on \texorpdfstring{$\CFA$}{the chain forest Hopf algebra}}

We briefly review the definitions of characters on Hopf monoids and Hopf algebras;
see~\cite[\S2.1]{AA}.  Let $\mathbf{H}$ be a connected Hopf monoid in vector species over a field $\Bbbk$ of characteristic~0, typically~$\Cc$.  A \defterm{character} $\zeta$ on $\mathbf{H}$ is a collection of linear maps $\zeta_I: \mathbf{H}[I] \to \Bbbk$ satisfying the following conditions:
\begin{enumerate}
    \item \textit{Naturality}: For each bijection $\sigma:I\to J$ and $x \in \mathbf{H}[I]$, we have $\zeta_J(\mathbf{H}[\sigma](x)) = \zeta_I(x)$.
    \item \textit{Multiplicativity}: For each $I = S \sqcup T$, $x \in\mathbf{H}[S]$, and $y \in\mathbf{H}[T]$, we have $\zeta_I(x\cdot y) = \zeta_S(x)\zeta_T(y)$.
    \item \textit{Unitality}: $\zeta_\0(1) = 1$.
\end{enumerate}
The characters on $\mathbf{H}$ form a group $\Xx(\mathbf{H})$ under the operation of \defterm{convolution}, defined by
\begin{equation}\label{eq:convolve-characters}
    (\chi*\phi)_I(x) = \sum_{I = S\sqcup T} \chi_S(x|_S)\phi_T(x/_S).
\end{equation}
The identity in $\Xx(\mathbf{H})$ is the \defterm{counit} $\epsilon$, which is the identity on $\mathbf{H}[\0]\isom\Bbbk$ and the zero map on $\mathbf{H}[I]$ for $I\neq\0$.  The inverse is given by the antipode map: $\chi^{-1}=\chi\circ\anti$.

By naturality, a character on a Hopf monoid $\mathbf{H}$ is essentially the same thing as a character on the Hopf algebra $\fock(\mathbf{H})$, and their character groups are isomorphic.   Accordingly, we will study characters on $\CFA$ rather than on $\CF$.

By multiplicativity, each character is determined by its values on the elements $\pha,\chain{1},\chain{2},\dots$, since they generate $\CFA$ as a free polynomial algebra.  For scalars $a,t_1,t_2,\ldots\in\Bbbk$, the character $\zeta_{a,\ttt}$ is defined by $\zeta_{a,\ttt}(\pha)=a$ and $\zeta_{a,\ttt}(\chain{n})=t_n$ for each $n$.  We start by computing the convolution of characters $\zeta_{a,\ttt}$ and $\zeta_{b,\sss}$.   For convenience, set $s_0=t_0=1$.
Applying the definition of convolution to~\eqref{eq:cg-alg-cprod-phantom} and~\eqref{eq:cg-alg-cprod-chain} gives
\begin{align}
(\zeta_{a,\ttt}*\zeta_{b,\sss})(\pha) &= a+b, \label{eq:char-convolve-phantom}\\
(\zeta_{a,\ttt}*\zeta_{b,\sss})(\chain{n}) &= s_n + \sum_{m=1}^n \sum_{j=0}^{n-m} \binom{n-m}{j} t_{j+1} s_{m-1} b^{n-m-j}. \label{eq:char-convolve}
\end{align}

In particular, the set
\[\Xx_0(\CFA)=\{\zeta_{a,\ttt}\in\Xx(\CFA)\colon a=0\}\]
is closed under convolution.  Moreover, it is closed under inversion, as can be seen either from the convolution formula, or by observing that $\anti(\pha)=-\pha$ by Prop.~\ref{prop:exorcism}, so that
\[\zeta_{a,\ttt}^{-1}(\pha) = \zeta_{a,\ttt}(S(\pha)) = -\zeta_{a,\ttt}(\pha) = -a.\]
It follows that $\Xx_0(\CFA)$ is a subgroup of $\Xx(\CFA)$; we call it the \defterm{exorcism group}.

\begin{theorem}\label{thm:power-series-subgroup}
The exorcism group is isomorphic to the group of formal power series of the form $1+\sum_{n=1}^\infty t_n x^n$, under multiplication.
\end{theorem}

\begin{proof}
When $a=b=0$, the convolution formula~\eqref{eq:char-convolve} reduces to $r_n= \sum_{m=0}^n s_{m}t_{n-m}$,
where as before $s_0 = t_0 = 1$.  On the other hand, this is also the formula for the coefficient of $x^n$ in the product of power series $\left(\sum_{n\geq0} s_n x^n\right)\left(\sum_{n\geq0} t_n x^n\right)$.  It follows that the function on $\Xx_0(\CFA)$ mapping $\zeta_{0,\ttt}$ to $\sum_{n\geq0} t_nx^n$ is a group isomorphism.
\end{proof}

This result is analogous to the fact that the character group of the Hopf monoid of permutahedra is isomorphic to the multiplicative group of \textit{exponential} formal power series  \cite[Thm.~2.2.2]{AA}.  In particular, the antipode formula for chains~\eqref{eq:antipode-for-chains} becomes a standard formula for inverting power series, as we now explain.  For $\zeta_{0,\ttt}\in\Xx_0(\CFA)$, we have by~\eqref{eq:antipode-for-chains}
\begin{align*}
\zeta^{-1}_{0,\ttt}(\chain{n}) = \zeta_{0,\ttt}(\anti(\chain{n}))
&= \sum_{\substack{V\subseteq[n]:\\ 1\in V}} (-1)^{n-|V|} \sum_{\substack{\Psi=\Psi_1|\cdots|\Psi_u\compn V\\ \text{natural}}} (-1)^u \zeta_{0,\ttt}(\chain{\lambda(\Psi)})\\
&= \sum_{\substack{\Psi=\Psi_1|\cdots|\Psi_u\compn [n]\\ \text{natural}}} (-1)^u \prod_{i=1}^u t_{|\Psi_i|}\notag\\
&= \sum_{\substack{(\alpha_1,\alpha_2,\dots,\alpha_k)\\ \alpha_i > 0\\ \sum_{\alpha_j} = n}} (-1)^k \prod_{i=1}^k t_{\alpha_i}.
\end{align*}
By Theorem~\ref{thm:power-series-subgroup}, this is the coefficient $s_n$ in the power series $1+\sum_{n=1}^\infty s_n x^n=\left(1+\sum_{n=1}^\infty t_n x^n\right)^{-1}$.
This formula can also be obtained by clearing denominators and solving for $s_n$ as a polynomial in the $t_n$, or by expanding $\left(1+\sum_{n=1}^\infty t_n x^n\right)^{-1}$ as a geometric series.

We conclude with a remark on the characters $\gamma_{u,r}=\zeta_{u,(r,r^2,r^3,\dots)}$, where $u,r\in\Rr$.  Convolution products of these ``geometric-series'' characters are given by the formulas
\begin{align}
(\gamma_{u,r}*\gamma_{v,q})(\pha)			&= u+v, \label{eq:mult-geom-phantom} \\
(\gamma_{u,r}*\gamma_{v,q})(\chain{n})		&= q^n + r\left(\frac{(r+v)^n-q^n}{r+v-q}\right) \notag \\
									&= H_n(r+v,q)-vH_{n-1}(r+v,q) \label{eq:mult-geom}
\end{align}
where $H_n(r,q) = \sum_{k=0}^n r^k q^{n-k}$.  Here~\eqref{eq:mult-geom-phantom} is immediate from~\eqref{eq:char-convolve-phantom}, and~\eqref{eq:mult-geom} follows by routine calculation from~\eqref{eq:char-convolve}.  Geometric-series characters do not form a subgroup of $\Xx(\CFA)$ (although those of the form $\gamma_{u,u}$ do form a group isomorphic to $\Cc$).  On the other hand, \eqref{eq:mult-geom-phantom} and~\eqref{eq:mult-geom} imply the curious identity
\[
\gamma_{1,r}*\gamma_{1,q} = \gamma_{1,q-1}*\gamma_{1,r+1}.
\]

\section{Open Questions} \label{ch:open-questions}

We conclude with some potential problems for future research.
\begin{enumerate}
\item What does the existence of the cancellation-free antipode formula for $\LOI$ say about the structure of posets?
\item (Suggested by Jos\'e Samper) Posets are in bijection not just with lattice of order ideals, but also with full-dimensional convex subfans of the braid arrangement, per the cone-preposet dictionary of \cite{GP}.  So $\LOI$ can be regarded as a Hopf monoid on such subfans.  Is there further interplay between the structure of $\LOI$ and subfan geometry?
\item Can any further progress be made on finding a cancellation-free antipode formula for the Hopf monoid of simplicial complexes, or of any of the other submonoids of $\SFam$ shown in Figure~\ref{fig:hierarchy}?
\item What else can be said about the Hopf algebra $\CFA$, regarded as a deformation of the Hopf algebra of symmetric functions via Theorem~\ref{thm:punchline}?
\end{enumerate}

\section*{Acknowledgements}

The authors thank Margaret Bayer, Bryan Gillespie, John Machacek, Alex McDonough, Jos\'e Samper,  Jacob White, and two anonymous referees for helpful discussions and suggestions.

\end{document}